\documentclass[11pt,reqno]{amsart}
\usepackage{amsfonts}
\usepackage{amsmath}
\usepackage{CJK}
\usepackage{amssymb}
\usepackage{latexsym,bm}
\usepackage[RGB]{xcolor}
\usepackage{mathrsfs}
\usepackage{amsthm}
\usepackage{hyperref}
\usepackage{ulem}
\numberwithin{equation}{section}
\newtheorem{theorem}{Theorem}[section]
\newtheorem{assumption}[theorem]{Assumption}
\newtheorem{lemma}[theorem]{Lemma}

\newtheorem{proposition}[theorem]{Proposition}
\newcommand{\dif}{\mathrm{d}}
\newcommand{\di}{\mathrm{div}}

\newcommand{\SUM}[3]{\sum\limits_{{#1}={#2}}^{#3}}

\newcommand{\intx}{\int_0^t \int_{\Omega} }

\begin{document}
\title[Spherically solutions for relaxed CNS system]{ Initial boundary value problems for 3-d Navier-Stokes equations with  hyperbolic heat conduction}
\author{Yuxi Hu and Reinhard Racke}
 \thanks{\noindent Yuxi Hu, Department of Mathematics, China University of Mining and Technology, Beijing, 100083, P.R. China, yxhu86@163.com\\
\indent Reinhard Racke, Department of Mathematics and Statistics, University of
Konstanz, 78457 Konstanz, Germany, reinhard.racke@uni-konstanz.de
}
\thanks{\noindent  }
\begin{abstract}
We study an initial boundary value problem for the 3-dimensional compressible Navier-Stokes equations with hyperbolic heat conduction, where the classical Fourier law is replaced by the Cattaneo-Christov constitutive relation. We focus on spherically symmetric solutions.   We establish the existence of uniform global small solutions to the resulting system. Furthermore, based on uniform a priori estimates, we rigorously justify both the relaxation limit and the vanishing viscosity limit.\\
{\bf Keywords}: Initial boundary value problem; hyperbolic heat conduction; spherically symmetric solution; global well-posedness; relaxation limit; vanishing viscosity limit\\
{\bf AMS classification code}: 35M13; 35Q35;
\end{abstract}
\maketitle

\section{Introduction}

The basic equations governing three-dimensional compressible fluid dynamics are given by
\begin{align}\label{1.1}
\begin{cases}
\partial_t\rho+\mathrm{div} (\rho u)=0,\\
\partial_t(\rho u) +\mathrm{div} (\rho u\otimes u)+\nabla p=\mathrm{div}\, S,\\
\partial_t (\rho {\mathcal E})+\mathrm{div}(\rho u {\mathcal E}+pu+q-Su) =0,
\end{cases}
\end{align}
for $(x,t)\in \Omega(\subset \mathbb R^3) \times(0,\infty)$. The quantities $\rho$, $u$, $p$, $\theta$, and $\mathcal{E}$ represent the fluid density, velocity, pressure, temperature, and specific energy, respectively. The quantities $q$ and $S$ denote the heat flux and stress tensor, respectively, and must be specified through constitutive relations to close the system \eqref{1.1}.

We assume the flow is Newtonian, which implies that the stress tensor $S$ is given by
\begin{align}\label{1.1a}
S=\mu(\theta)\left(\nabla u+(\nabla u)^T-\frac{2}{3} \di u I_3\right)+\lambda(\theta) \di u I_3,
\end{align}
where $\mu$ and $\lambda$ are the shear and bulk viscosities, respectively, and both are functions of the temperature $\theta$.

The formulation of the heat flux $q$ is more intricate. Instead of employing the classical Fourier law for heat conduction, we adopt the Cattaneo-Christov (CC) constitutive relation:
\begin{align}\label{1.1b}
\tau (\theta) \rho \left(\partial_t q+u\cdot\nabla q-q\cdot \nabla u+(\di u)q\right)+q+\kappa (\theta)\nabla \theta=0,
\end{align}
where $\tau(\theta)$ is the relaxation parameter accounting for the time lag in the response of heat flux to the temperature gradient, and $\kappa(\theta)$ is the heat conductivity coefficient. Both parameters depend on the temperature.

The Cattaneo-Christov relation \eqref{1.1b} was originally proposed  in the linearized case by Cattaneo \cite{CA} to resolve the paradox of infinite heat propagation speed implied by Fourier's law, and was later extended to an objective (frame-indifferent) nonlinear formulation by Christov \cite{CHR}. The CC model has been shown to be particularly relevant in extreme physical scenarios, such as instantaneous heating by laser pulses \cite{LE08}, heat conduction in nanometer-scale materials \cite{JSA11}, and the Rayleigh-B\'enard convection near MEMS devices \cite{GL03}. In fact, the dimensionless Cattaneo number, defined by $C=\frac{\tau \kappa}{L^2}$ where $L$ is a reference length scale, is used to quantify the significance of the CC model relative to Fourier’s law \cite{KH15}. When $\kappa$ is large or $L$ is small, the CC model can become particularly important.

Note that the second law of thermodynamics may not hold for the new CC model if the classical thermodynamic equation remains unchanged; see \cite{CHO, HR20, HuRa024, MO17}. Inspired by Coleman \textit{et al.} \cite{CHO}, we assume that the specific total energy is given by
\begin{align}\label{1.2}
{\mathcal E}=\frac{1}{2}u^2+e,
\end{align}
with the specific internal energy $e$ and the pressure $p$ given by
\begin{align}\label{1.3}
e=C_v \theta+a(\theta)q^2, \quad p=R \rho \theta,
\end{align}
where $a(\theta)=\frac{Z(\theta)}{\theta}-\frac{1}{2} Z^\prime(\theta)$ and $Z(\theta)=\frac{\tau(\theta)}{\kappa(\theta)}$.
$C_v$ and $R$ denote the heat capacity at constant volume and the gas constant, respectively. The quantities $p$ and $e$ satisfy the usual thermodynamic relation
$$
\rho^2 e_\rho=p-\theta p_\theta.
$$

For \(\tau(\theta) = 0\), the system \eqref{1.1}-\eqref{1.3} reduces to the classical compressible Navier-Stokes-Fourier system, for which the global well-posedness, blow-up of smooth solutions, and large-time behavior have been extensively studied. In particular, the local existence and uniqueness of smooth solutions were established by Serrin \cite{SE}, Nash \cite{NA} and Itaya \cite{ITA70} for initial data away from vacuum. Later, Matsumura and Nishida \cite{MN} obtained global smooth solutions for small initial data without vacuum. For large initial data, Xin \cite{X}, and Cho and Jin \cite{CJ}, showed that smooth solutions must blow up in finite time if the initial data contains vacuum. See \cite{DA1, DA2, LI, JZ01, JZ03, FE} for results on the global existence of weak solutions.

On the other hand, for \(\tau(\theta) > 0\), there are relatively few results. The authors  first studied in \cite{HR16} the well-posedness of strong solutions in Sobolev spaces, as well as the relaxation limit, under a linearized form of \eqref{1.1b}. The decay properties of the resulting solutions were established in \cite{LW23,TZZ24}. Crin-Barat, Kawashima, and Xu \cite{CKX24} obtained global-in-time well-posedness for small initial data in Besov spaces, along with a strong global relaxation limit. In this sense, the Cauchy problem for the system \eqref{1.1}-\eqref{1.3} has been extensively investigated, cf. the survey on our results in \cite{Ra025-cam}. However, to the best of our knowledge, initial boundary value problems have not yet been addressed.  We note that the initial boundary value problem  for compressible Navier-Stokes equations with Maxwell's law -- i.e. relaxation for the stress tensor -- has been considered recently, see \cite{HL25, HY25, HZ25}. 
As a first step in this direction, we study the initial boundary value problem for system \eqref{1.1}-\eqref{1.3} under the assumption of spherical symmetry.
We consider spherically symmetric solutions of the following form:
\begin{align}\label{1.4}
\rho(t,x)=\rho(t,r), \quad u(t,x)=u(t, r)\frac{x}{r}, \quad \theta(t,x)=\theta(t,r), \quad q(t,x)=q(t,r)\frac{x}{r},
\end{align}
where $x\in \Omega ;= \overline{B(0,1)}^c$ or $\Omega := B(0,2)\setminus \overline{B(0,1)}$, with $B(0,r)$ denoting the open ball of radius $r$ centered at the origin. Under this symmetry, the system \eqref{1.1}-\eqref{1.3} reduces to:
\begin{align}\label{1.5}
\begin{cases}
	\rho_t+(\rho u)_r+\frac{2}{r} \rho u = 0, \\
	\rho u_t+\rho u u_r+ p_r =\left( \left(\frac{4}{3} \mu(\theta)+\lambda(\theta)\right)\left(u_{r}+\frac{2}{r}u \right)\right)_r ,\\
	\rho e_\theta \theta_t + (\rho u e_\theta - \frac{2a(\theta)}{Z(\theta)} q)\theta_r + p\left(u_r+\frac{2}{r} u\right) + q_r + \frac{2}{r} q = \\
	\quad  \left(\frac{2}{\tau(\theta)}+\frac{4}{r} u\right) a(\theta) q^2 + \mu(\theta)\left( 2u_r^2+\frac{4}{r^2}u^2 - \frac{2}{3} \left(u_r+\frac{2}{r} u\right)^2\right) + \lambda(\theta)\left(u_r+\frac{2}{r} u\right)^2, \\
	\tau (\theta) \rho \left(q_t+u q_r+\frac{2}{r} u q\right) + q + \kappa(\theta) \theta_r = 0.
\end{cases}
\end{align}

We supplement system \eqref{1.5} with the initial data
\begin{align}\label{initial}
(\rho, u, \theta, q)|_{t=0} = (\rho_0, u_0, \theta_0, q_0),
\end{align}
possibly depending on $\tau$, and the boundary conditions
\begin{align}\label{1.6}
u|_{\partial \Omega} = q|_{\partial \Omega} = 0,
\end{align}
with the domain $\Omega$ reduced to either $(1, \infty)$ or to $(1, 2)$.

To establish estimates independent of $\tau(\theta)$, $\mu(\theta)$, and $\lambda(\theta)$ and to facilitate taking the limit, we assume the following relations:
\[
\tau(\theta) = \tau g(\theta), \quad \mu(\theta) = \mu h(\theta), \quad \lambda(\theta) = \lambda l(\theta),
\]
where $g$, $h$, and $l$ are smooth (and positive) functions of $\theta$, and $\tau$, $\mu$, and $\lambda$ are positive constants. The following assumptions are used throughout the paper:

\begin{assumption}\label{ass}
\mbox{}\\
(1) The initial and boundary data satisfy the usual compatibility conditions:
\begin{align}\label{compatibility condition}
	\partial_t^k u(0,x)\big|_{\partial\Omega} = \partial_t^k q(0,x)\big|_{\partial\Omega} = 0, \quad k = 0, 1,
\end{align}
where $\partial_t u(0,x)$ and $\partial_t q(0,x)$ are defined recursively using equations $\eqref{1.5}_2$ and $\eqref{1.5}_4$, respectively.\\[0.5em]
(2) The possible dependence of the initial data on $\tau$ is compatible in the following sense:
\begin{align}\label{compatibility}
\left\|r\left(q_0 + \kappa(\theta_0)(\theta_0)_r\right)\right\|_{H^1} = O(\sqrt{\tau}), \quad \text{as } \tau \to 0.
\end{align}
\end{assumption}

To formulate our theorem, we introduce the following energy functional:
\begin{align*}
E(t) := & \sup_{0 \le s \le t} \|r(\rho - 1, u, \theta - 1, q)\|_{L^2}^2 + \|r(\rho_r, u_r, \theta_r, q_r)\|_{L^2}^2 + \|(u, q)\|_{L^2}^2+ \|(u_r, q_r)\|_{L^2}^2   \\
& + \|r(\rho_{rr}, u_{rr}, \theta_{rr}, q_{rr})\|_{L^2}^2  + \|r(\rho_t, u_t, \theta_t, \sqrt{\tau} q_t)\|_{L^2}^2  + \|r(\rho_{tr}, u_{tr}, \theta_{tr}, \sqrt{\tau} q_{tr})\|_{L^2}^2 \\
&+ \tau^2 \|r(\rho_{tt}, u_{tt}, \theta_{tt}, \sqrt{\tau} q_{tt})\|_{L^2}^2+ \left(\frac{4}{3}\mu+\lambda\right)^2 \|r(u_{rrr}, \tau u_{trr})\|_{L^2}^2
\end{align*} The corresponding dissipation functional is defined as:
\begin{align*}
\mathcal{D}(t) := & \|r q\|_{L^2}^2 + \|r D(\rho, u, \theta, q)\|_{L^2}^2 + \|(u, q)\|_{L^2}^2 + \|r(\rho_{rr}, u_{rr}, \theta_{rr}, q_{rr})\|_{L^2}^2  \\
& + \|(u_r, q_r, u_t, q_t)\|_{L^2}^2 + \|r(\rho_{tr}, u_{tr}, \theta_{tr}, q_{tr})\|_{L^2}^2 + \|r(\rho_{tt}, u_{tt}, \theta_{tt})\|_{L^2}^2   \\
& + \tau^2 \|r q_{tt}\|_{L^2}^2+\left( \frac{4}{3} \mu+\lambda \right)\| r (u_{trr}, u_{ttr})\|_{L^2}^2
\end{align*}
where \( D := (\partial_t, \partial_r) \).

Our first result concerns the global-in-time small well-posedness of the initial boundary value problem \eqref{1.5}-\eqref{1.6}, with estimates that are uniform with respect to the parameters \(\tau\), \(\lambda\), and \(\mu\).
\begin{theorem}\label{th1.1}
Let Assumption \ref{ass} hold and $r(\rho_0-1, \theta_0-1, q_0)\in H^{{3}}$,  $ru_0\in H^4$. Then there exists a small constant \(\epsilon_0 > 0\) such that if
\[
\|r(\rho_0-1, \theta_0-1, q_0)\|_{H^{3}}^2+\|ru_0\|_{H^4}^2  \le \epsilon_0,
\]
there exists a unique global solution \((\rho, u, \theta, q)\) to the initial boundary value problem \eqref{1.5}-\eqref{1.6} such that
\[
(\rho-1, \theta-1, q)\in \bigcap\limits_{k=0}^2 C^k([0,\infty), H^{2-k}),\, u\in \bigcap\limits_{k=0}^1 C^k([0,\infty), H^{3-k}) \bigcap C^2([0,\infty), L^2)
\]
and satisfy the uniform estimate
\[
E(t)+\int_0^\infty \mathcal D(t) \dif t \le C E(0),
\]
where \(C\) is a constant independent of the parameters \(\tau\), \(\lambda\), and \(\mu\).
\end{theorem}

Building on the uniform estimates established in Theorem \ref{th1.1}, we now state our second result regarding the vanishing viscosity limit.
\begin{theorem}\label{th1.2}
 Fix $\tau>0$. Let $\varepsilon=( \mu, \lambda)$ and $(\rho_\tau^\varepsilon, u_\tau^\varepsilon, \theta_\tau^\varepsilon, q_\tau^\varepsilon)$ be the global solutions obtained in Theorem \ref{th1.1}, then there exists
$( \rho_\tau^0,  u_\tau^0,  \theta_\tau^0, q_\tau^0)\in L^\infty((0,\infty); H^2 ) \cap C^0((0,\infty); H^{2-\delta}) $ for any $\delta>0$, such that, as $\varepsilon\rightarrow 0$, up to subsequences, for any $T>0$,
 \begin{align}
(\rho_\tau^\varepsilon, u_\tau^\varepsilon, \theta_\tau^\varepsilon, q_\tau^\varepsilon) \rightarrow (\rho_\tau^0, u_\tau^0, \theta_\tau^0, q_\tau^0) \quad \mbox{\rm  strongly\quad in } \quad C([0,T],H_{loc}^{2-\delta} )
\end{align}
where $( \rho_\tau^0,  u_\tau^0,  \theta_\tau^0, q_\tau^0)$ is a strong solution to the three-dimensional hyperbolized compressible Euler-Cattaneo-Christov system \eqref{4.3} (below), satisfying the initial and boundary conditions \eqref{initial} and \eqref{1.6}.
\end{theorem}

The next result shows the global relaxation limit.
\begin{theorem}\label{th1.3}
Fix $\mu>0, \lambda>0$. Let $\varepsilon=( \mu, \lambda)$  and  $(\rho_\tau^\varepsilon, u_\tau^\varepsilon, \theta_\tau^\varepsilon, q_\tau^\varepsilon)$  be the global solutions obtained in Theorem \ref{th1.1}, then there exists
$(\rho_0^\epsilon-1, \theta_0^\epsilon-1, q_0^\epsilon)\in L^\infty((0,\infty); H^2 ) \cap C^0((0,\infty); H^{2-\delta} )$ and $u_0^\epsilon \in  L^\infty((0,\infty); H^3 ) \cap C^0((0,\infty); H^{3-\delta} )$for any $\delta>0$, such that, as $\tau\rightarrow 0$, up to subsequences, for any $T>0$,
 \begin{align}
(\rho_\tau^\varepsilon, \theta_\tau^\varepsilon, q_\tau^\varepsilon) \rightarrow (\rho_0^\epsilon, \theta_0^\epsilon, q_0^\epsilon) \quad \mbox{\rm  strongly\quad in } \quad C([0,T],H_{loc}^{2-\delta} )
\end{align}
and
\[
 u_\tau^\varepsilon  \rightarrow u_0^\epsilon \quad \mbox{\rm  strongly\quad in } \quad C([0,T],H_{loc}^{3-\delta} )
 \]
where $q_0^\epsilon=-\kappa(\theta_0^\epsilon) \partial_r \theta_0^\epsilon, a.e., $ and $( \rho_0^\epsilon, u_0^\epsilon, \theta_0^\epsilon)$ is a strong solution to the three-dimensional compressible Navier-Stokes-Fourier system in spherical symmetry \eqref{4.5} (below), satisfying the initial and boundary conditions \eqref{initial} and \eqref{1.6} with $q_0=-\kappa(\theta_0)(\theta_0)_r$.
\end{theorem}
The paper is organized as follows. In Section 2 we state the local existence theorem for the system \eqref{1.5}-\eqref{1.6}. The central Section 3 provides a priori estimates, being necessary for extending a local solution to a global one and for carrying out the singular limits. In Section 4 the proofs the main theorems are given.

Finally, we introduce some notation. $W^{m,p}=W^{m,p}(\Omega),0\le m\le \infty,1\le p\le \infty$, denotes the usual Sobolev space with norm $\left\| \cdot \right\|_{W^{m,p}}$, $H^m$ and $L^p$ stand for $ W^{m,2}$ resp. $W^{0,p}$.

 \section{Local well-posedness}

In this section, we establish the local well-posedness for system \eqref{1.5}-\eqref{1.6}. We note that, to our knowledge, general results for the initial boundary value problem of hyperbolic-parabolic coupled systems are lacking. To obtain the desired well-posedness, we follow the approach in \cite{Zheng} (pp. 119-124).
First, Let $M_0:=\SUM k 0 2 \| D^k(\rho-1, u, \theta-1, q)|_{t=0}\|_{L^2}^2 <\infty$. Let $M_1$ be a positive constant such that when $\sup_{0\le t\le h}\|(\rho-1, u, \theta-1, q)\|_{H^2} \le M_1$,
\[
\min \{ \min_{[0,h]\times \Omega} \rho(t,r), \min_{[0,h]\times \Omega}  \theta (t,r), \min_{[0,h]\times \Omega}  e_{\theta}(t,r)\} =:\gamma >0.
\]

Introduce  the space
\begin{align*}
X_h(M_2, M_3)=\left\{ (\rho, u, \theta, q)| (\rho-1, \theta-1, q)\in \bigcap\limits_{k=0}^2 C^k([0,h], H^{2-k}), u\in C([0,h], H^3),\right.\\
u_t\in C([0,h], H^2), u_{tt} \in C([0,h], L^2) \cap L^2([0,h], H^1))\\
\sup_{0\le t\le h} \SUM k 0 2 \|D^k(\rho-1, u, \theta-1, q)\|_{L^2}^2 \le M_2,\\
\sup_{0\le t \le h} (\|u_{xxx}\|_{L^2}^2+\|u_{txx}\|_{L^2}^2 )+\int_0^h \|u_{xtt}\|_{L^2}^2 \dif t \le M_3\\
\left.(\rho, u, \theta, q)|_{t=0}=(\rho_0, u_0, \theta_0, q_0),\quad (u,q)|_{\partial \Omega}=0, \right\}
\end{align*}
where $M_2<M_1$, $M_3$ are positive constant specified in the course of the proof.

Now, we are ready to state the local well-posedness theorem.

\begin{theorem}\label{loc}
Assume $(\rho_0-1, \theta_0-1, q_0)\in H^{{3}}$,  $u_0\in H^4$ and the compatibility condition \eqref{compatibility condition} hold. Then,
there exists a positive constant $\epsilon_0$ and $t_0$ depending on $\epsilon_0$ such that when $M_0\le \epsilon_0$, problem  \eqref{1.5}-\eqref{1.6} admits a unique local solution
$(\rho, u, \theta, q)\in X_{t_0}(C_1 M_0, C_2M_0)$ where $C_1, C_2$ only depends on $\gamma$.
\end{theorem}

 \begin{proof}
We present a concise proof of this theorem, though detailed proofs for similar systems can be found in \cite{Zheng} and references therein. Specifically, we divide the proof into four steps.

 Step 1: For any $(\tilde \rho, \tilde u, \tilde \theta, \tilde q)\in X_h(M_2, M_3)$, consider the auxiliary linear problem
 \begin{align}\label{loc1}
 \begin{cases}
 \tilde \rho \tilde e_{\theta} \theta_t +(\tilde \rho \tilde u \tilde e_{\theta}-\frac{2 a(\tilde \theta)}{Z(\tilde \theta)} \tilde q) \theta_r+q_r=f_1,\\
 \tau(\tilde \theta)\tilde \rho (q_t+\tilde u q_r)+\kappa(\tilde \theta) \theta_r=f_2,\\
 (\theta, q)|_{t=0}=(\theta_0, q_0),\quad q|_{\partial \Omega}=0,
 \end{cases}
 \end{align}

 \begin{align}\label{loc2}
 \begin{cases}
  \tilde \rho u_t-(\frac{4}{3} \mu(\tilde \theta)+\lambda(\tilde \theta))u_{rr}=f_3,\\
  u|_{t=0}=u_0, u|_{\partial \Omega}=0,
 \end{cases}
 \end{align}

 and
 \begin{align}\label{loc3}
 \begin{cases}
  \rho_t+\tilde u \rho_r=f_4,\\
  \rho|_{t=0}=\rho_0,
  \end{cases}
  \end{align}
 where
 \begin{align*}
 \begin{cases}
 f_1=\left(\frac{2}{\tau(\tilde\theta)}+\frac{4}{r}\tilde u\right) a(\tilde\theta) \tilde q^2 + \mu(\tilde \theta)\left( 2\tilde u_r^2+\frac{4}{r^2}\tilde u^2 - \frac{2}{3} \left(\tilde u_r+\frac{2}{r} \tilde u\right)^2\right) \\
 \qquad\qquad\qquad\qquad\qquad\qquad\qquad\quad \qquad+ \lambda(\tilde \theta)\left(\tilde u_r+\frac{2}{r} \tilde u\right)^2-(\tilde u_r+\frac{2}{r} \tilde u) \tilde p-\frac{2}{r} \tilde q,\\
f_2=-\tau(\tilde \theta) \tilde \rho \cdot \frac{2}{r} \tilde u \tilde q-\tilde q, \\
f_3= \left(\frac{4}{3} \mu(\tilde \theta)+\lambda(\tilde \theta)\right)\left(\frac{2}{r}\tilde u_r-\frac{2}{r^2} \tilde u\right) + \lambda^\prime(\tilde \theta)\tilde \theta_r\left(\tilde u_r+\frac{2}{r}\tilde u\right) + \frac{4}{3}\mu^\prime(\tilde \theta)\tilde  \theta_r\left(\tilde u_r-\frac{\tilde u}{r}\right)\\
\qquad\qquad\qquad\qquad\qquad\qquad\qquad\qquad\qquad\qquad\qquad\qquad-R \tilde \rho \tilde \theta_r-R\tilde \theta \tilde \rho_r - \tilde \rho \tilde u  \tilde u_r,\\
f_4= -\frac{2}{r} \tilde \rho \tilde u-\tilde \rho \tilde u_r.\\
\end{cases}
 \end{align*}

 We deal with the system \eqref{loc1}, \eqref{loc2}, \eqref{loc3}, respectively. First, problem \eqref{loc1} is a linear symmetric hyperbolic system of first order (by dividing $\eqref{loc1}_2$ by $\kappa(\tilde \theta)$).  We first check that the boundary condition $q|_{\partial\Omega}=0$ is non-characteristic and satisfy the maximal nonnegative condition (admissible in the sense of Friedrich, see \cite{Zheng, Friedrichs}). Indeed,  we rewrite system \eqref{loc1} with $U=(\theta, q)$ as
 \begin{align}
 A^0 U_t+A^1 U_r=F,
 \end{align}
 where
 \[
 A^0=
 \begin{pmatrix}
 \tilde \rho \tilde e_\theta &0\\
 0& \frac{\tau(\tilde \theta) \tilde \rho}{\kappa (\tilde \theta)}
 \end{pmatrix},
 A^1=
 \begin{pmatrix}
 \left(\tilde \rho \tilde u \tilde e_\theta-\frac{2a(\tilde \theta)}{Z(\tilde \theta)} \tilde q\right) & 1\\
 1&  \frac{\tau (\tilde \theta) \tilde \rho \tilde u}{\kappa (\tilde \theta)}
 \end{pmatrix}, F=\mathrm{diag }\{f_1, f_2\}.
 \]
 The boundary condition reduces to
 \[
 P U|_{\partial \Omega}=0,\,  \text{with} \,  P=
 \begin{pmatrix}
 0&0\\
 0&1
 \end{pmatrix}.
 \]
 Note that by virtue of $(\tilde u, \tilde q)|_{\partial \Omega}=0$, we have $\det \left( (A^0)^{-1} A^1\right)|_{\partial \Omega} \neq 0$. Consequently, the boundary condition $q|_{\partial \Omega}=0$ is non-characteristic. Now, we show the boundary condition is maximally nonnegative, i.e., the matrix $A^1 \cdot \nu|_{\partial\Omega}$ is positive semidefinite on the null space $K$ of $P$ but not on any space containing $K$. Here $\nu=-1$. Let $\xi=(\xi_1,0)^T \in K = \mathrm{span} \{(1,0)^T\}$, then,
\begin{align*}
\xi^TA^1\cdot\nu|_{\partial\Omega}\xi=0.
\end{align*}
On the other hand, $\mathbb R^2$ is the only space containing $K$ as a proper subspace.  Take $\psi=(1,1)^T$, we get
\begin{align*}
\psi ^TA^1\cdot\nu|_{\partial\Omega}\psi=-2<0.
\end{align*}
Thus, the maximally nonnegative property is satisfied.  Furthermore, since the coefficients belong to $\bigcap\limits_{k=0}^2 C^k ([0,h], H^{2-k})$ and $D^2 f_1 \in L^2([0,h], L^2)$ and the compatibility condition satisfied up to first-order, by Theorem 1.3.10 in \cite{Zheng} (pp. 18-19) (originally coming from \cite{RM74}) , problem \eqref{loc1} admits a unique solution $(\theta-1, q)\in \bigcap\limits _{k=0}^2C^k([0,h], H^{2-k})$ satisfying
 \begin{align*}
& \SUM k 0 2 \|D^k(\theta-1, q)\|_{L^2}^2 \\
 &\le C \left( \SUM k 0 2 \|D^k (\theta-1, q)|_{t=0}\|_{L^2}^2+\SUM k 0 1 \|D_t^k f_1|_{t=0}\|_{H^{1-k}}^2 + t \int_0^t \SUM l  0 2 \|D^l f_1\|_{L^2}^2 \dif t \right) e^{C M_2 t},
 \end{align*}
 where $C$ is a positive constant depending on $\gamma$.

 Concerning the initial-boundary value problem for linear parabolic system \eqref{loc2}, we deduce from energy methods that the unique solution $u$ satisfies
 \begin{align*}
 \SUM k 0 2  \|D^k  u\|_{L^2}^2+\int_0^t \|(u_x, u_{xx}, u_{tx}, u_{ttx}, u_{txx})\|_{L^2}^2 \dif t
 \le C e^{C M_2 t} \left\{ \SUM k 0 2 \|D^k  u|_{t=0}\|_{L^2} +M_2 t \right\},\\
 \|u_{xxx}\|_{L^2}^2+\| u_{txx}\|_{L^2}^2\le C(M_2)(1+ \SUM k 1 2 \|D^k u\|_{L^2}^2 ),
 \end{align*}
 where $C(M_2)$ is a positive constant depending on $M_2$.

 Finally, for the transport equation $\eqref{loc3}_1$, we define the characteristic line
 \[
 \frac{d X(t,r)}{d t}= \tilde u(t, X(t,r)), \quad X(0,r)=r.
 \]
 Then, we can derive an explicit  solution for \eqref{loc3},
 \begin{align}
 \rho(t,r)=\rho_0(r_0)+\int_0^t f_4(s, X(s, r_0)) \dif s ,
 \end{align}
 where $r_0$ denotes the unique point on the $r$-axis ($t=0$) such that the characteristic line passing through $(0, r_0)$ and $(t, r)$. Moreover,
 from usual energy estimates, we can easily get
  \begin{align*}
& \SUM k 0 2 \|D^k(\rho-1)\|_{L^2}^2 \\
 &\le C \left( \SUM k 0 2 \|D^k (\rho-1)|_{t=0}\|_{L^2}^2+\SUM k 0 1 \|D_t^k f_4|_{t=0}\|_{H^{1-k}}^2 + t \int_0^t \SUM l  0 2 \|D^l f_4\|_{L^2}^2 \dif t \right) e^{C M_2 t}.
 \end{align*}

Step 2: We first choose
 \begin{align*}
 M_2
 &=2C \left( \SUM k 0 2 \|D^k (\rho-1, \theta-1, q)|_{t=0}\|_{L^2}^2+\SUM k 0 1 \|D_t^k (f_1, f_4)|_{t=0}\|_{H^{1-k}}^2+\SUM k 0 2 \|D^k  u|_{t=0}\|_{L^2}  \right) \\
 &=:C_1 M_0.
 \end{align*}
 Then, choose
 \begin{align*}
 M_3=\max \{M_2,  C(M_2)(1+M_2)\}=:C_2 M_0.
 \end{align*}
 Finally, choose $h$ depending on $M_0$ such that for $t\le h$,
 \begin{align*}
 C  t  e^{C M_2 t}  \int_0^t \SUM l  0 2 \|D^l (f_1, f_4)\|_{L^2}^2 \dif t  \le \frac{M_2}{2}, \quad
 e^{C M_2 t}  \le \frac{5}{4},  \quad
 C e^{C M_2 t}  t \le \frac{1}{4}.
 \end{align*}
 Therefore, the linear operator defined by \eqref{loc1}-\eqref{loc3} maps $X_h(M_2, M_3)$ into itself.

Step 3: Following the above two steps, we can get an iterative sequence $(\rho_n, u_n, \theta_n, q_n)\in X_h(M_2, M_3)$. By usual compactness
 argument, the sequence$(\rho_n-1, \theta_n-1, q_n)$ converge in $C([0,h], H^1)\cap C^1([0,h], L^2)$ and $u_n$ converge in $C([0,h], H^2)\cap C^1([0,h],L^2)\cap L^2([0,h], H^3)$ and the limit
 function $(\rho, u, \theta, q)$ satisfy the system \eqref{1.5}-\eqref{1.6}. Moreover, the function $(\rho, u, \theta, q)$ belongs to the class: for $0\le k\le 2$,
 \[
 D^k(\rho-1, u, \theta-1, q)\in L^\infty((0,h), L^2),\,  u_{xxx}, u_{txx}\in L^\infty((0,h), L^2), u_{xtt}\in L^2((0,h), L^2)
 \]
 and is the unique pair of solutions for problem \eqref{1.5}-\eqref{1.6} in this class.

Step 4:  Substituting this $(\rho, u, \theta, q)$ into the coefficients in \eqref{1.5}, the reduced linearized problems admits a unique pair of solutions in
 $X_h(M_2,M_3)$. By uniqueness, $(\rho, u, \theta, q) \in X_h(M_2,M_3)$.

  Therefore, by choosing $M_0$ sufficiently small so that $M_2:=C_1 M_0 \le M_1$, the proof of the theorem is complete.

 \end{proof}

 \section{A priori estimates}

In order to extend the local solution obtained in Theorem \ref{loc} to a global solution, we need the following a priori estimates.
\begin{proposition}\label{pro}
Let Assumption \ref{ass} hold and $(\rho, u, \theta, q)$ be the local solution to system $\eqref{1.5}-\eqref{1.6}$ given in Theorem \ref{loc}. Then, there exists a $\delta>0$ such that if $E(t)<\delta$,
\begin{align}\label{hu3.1}
E(t)+\int_0^t \mathcal D(s)\dif s \le C E(0),
\end{align}
where $C$ is a constant independent of $\tau, \lambda, \mu$.
\end{proposition}

%
%

Without loss of generality, we assume $g(1)=\kappa(1)=h(1)=l(1)=1$, $\tau \le 1$.  Moreover, there exists  $\delta_1$ such that if $E(t)\le \delta_1$, one has
\begin{align}
&\frac{3}{4}\le \rho(t,x),\theta(t,x)\le\frac{5}{4}, \label{new1}\\
&2C_v\ge e_{\theta}\ge \frac{C_v}{2},\, 2\tau \ge Z(\theta)=\frac{\tau   (\theta)}{\kappa(\theta)}\ge \frac{\tau   }{2}, |e_{\theta\theta}|+|e_{\theta q}|+|e_{\theta \theta\theta}|+
|e_{\theta \theta q}|+|e_{\theta q q}| \le C, \label{new2}\\
&a(\theta)+\left(\frac{Z(\theta)}{2\theta}\right)^\prime=\frac{1}{\theta}(1-\frac{1}{2\theta}) Z(\theta)+\frac{1}{2}(\frac{1}{\theta}-1) Z^\prime(\theta) \ge \frac{1}{4} \tau ,   \label{new3}
\end{align}
where $ C$ denotes a universal constant which is independent of $\tau  $, $\mu$ and $\lambda$. Note that \eqref{new3} can be satisfied by choosing $\theta$ sufficiently close to $1$.

The proof of Proposition \ref{pro} is divided into the following Lemmas.
\begin{lemma}\label{le1}
There exists some constant $C$ such that
\begin{align} \label{1.7}
\int_{\Omega}r^2\left((\rho-1)^2+u^2+ (\theta-1)^2+\tau   q^2\right)\dif r+\int_{0}^{t}\int_{\Omega} r^2 q^2 \dif r \dif s
\le CE_0.
\end{align}
\end{lemma}
\begin{proof}
From $\eqref{1.5}_{3, 4}$, we get
\begin{align}\label{1.8}
\rho e_t+\rho u e_r+p(u_r+\frac{2}{r} u)+q_r+\frac{2}{r} q=\mu(\theta)\left( 2u_r^2+\frac{4}{r^2}u^2-\frac{2}{3} \left(u_r+\frac{2}{r} u\right)^2\right) \nonumber
\\+\lambda(\theta)\left(u_r+\frac{2}{r} u\right)^2. \qquad\qquad
\end{align}
Dividing the above equation by $\theta$, and using formula \eqref{1.3}, one has
\begin{align*}
\frac{\rho}{\theta}(C_v\theta+a(\theta)q^2)_t+&\frac{\rho u}{\theta} (C_v \theta+a(\theta) q^2)_r+R \rho (u_r+\frac{2}{r} u)+\frac{1}{\theta}(q_r+\frac{2}{r}q)=\\
& \frac{\mu(\theta)}{\theta} \left( 2u_r^2+\frac{4}{r^2}u^2-\frac{2}{3} \left(u_r+\frac{2}{r} u\right)^2\right)+\frac{\lambda (\theta)}{\theta} \left(u_r+\frac{2}{r} u\right)^2.
\end{align*}
For the above equations, first, we have
\begin{align*}
\frac{1}{\theta}C_v\theta_t=C_v(\ln\theta)_t,
\end{align*}
and
\begin{align*}
&\frac{\rho}{\theta} (a(\theta)q^2)_t+\frac{\rho u}{\theta} (a(\theta) q^2)_r\\
&=\rho\left( \frac{a(\theta) }{\theta} q^2\right)_t+\rho \frac{\theta_t}{\theta^2} a(\theta) q^2 +\rho u \left( \frac{a(\theta)}{\theta} q^2 \right)_r+\rho u \frac{\theta_r}{\theta^2} a(\theta) q^2\\
&=\rho\left( \frac{a(\theta)}{\theta} q^2\right)_t-\rho \left( \frac{Z(\theta)}{2\theta^2}\right)_t q^2+\rho u \left( \frac{a(\theta)}{\theta} q^2 \right)_r-\rho u \left(\frac{Z(\theta)}{2\theta^2}\right)_r q^2\\
&=\rho \left( \frac{a(\theta)}{\theta} q^2 \right)_t-\rho \left(\frac{Z(\theta)}{2\theta^2}q^2\right)_t+\rho \frac{Z(\theta)}{\theta^2} \left(\frac{1}{2}q^2\right)_t\\
&\qquad\qquad\qquad+\rho u \left( \frac{a(\theta)}{\theta}q^2\right)_r-\rho u \left(\frac{Z(\theta)}{2\theta^2}q^2 \right)_r+\rho u \frac{Z(\theta)}{\theta^2} \left(\frac{1}{2} q^2\right)_r\\
&=\rho \left( \left(\frac{a(\theta)}{\theta}-\frac{Z(\theta)}{2\theta^2}\right) q^2\right)_t+\rho u \left(  \left(\frac{a(\theta)}{\theta}-\frac{Z(\theta)}{2\theta^2}\right) q^2\right)_r+\frac{Z(\theta)}{\theta^2} q (\rho q_t+\rho u q_r),
\end{align*}
where we have used the identity
$$
\left(\frac{Z(\theta)}{2\theta^2}\right)^\prime =-\frac{a(\theta)}{\theta^2}.
$$
Using the equation $\eqref{1.5}_4$, one has
\begin{align*}
\frac{Z(\theta)}{\theta^2} q (\rho q_t+\rho u q_r)
=\frac{1}{\kappa(\theta)\theta^2} \left( -\tau(\theta) \rho u \frac{2}{r}-1\right)q^2-\frac{\theta_r}{\theta^2} q.
\end{align*}
Thus, we derive that
{\small
\begin{align*}
&\rho( C_v \ln \theta +A(\theta) q^2 )_t+\rho u (C_v \ln \theta+A(\theta) q^2)_r+R \rho (u_r+\frac{2}{r} u)+\left(\frac{q}{\theta}\right)_r+\frac{2}{r} \frac{q}{\theta}\\
&=\frac{1}{\kappa(\theta) \theta^2} \left(1+\tau(\theta) \rho u \frac{2}{r}\right) q^2+ \frac{\mu(\theta)}{\theta} \left( 2u_r^2+\frac{4}{r^2}u^2-\frac{2}{3} \left(u_r+\frac{2}{r} u\right)^2\right)+ \frac{\lambda(\theta)}{\theta} \left(u_r+\frac{2}{r} u\right)^2,
\end{align*}
}
where
\[
A(\theta):=\frac{a(\theta)}{\theta}-\frac{Z(\theta)}{2\theta^2}=-\left(\frac{Z(\theta)}{2\theta}\right)^\prime.
\]
Multiplying the above equation by $r^2$, it yields{\small
\begin{align*}
&\left( r^2 \rho \left( C_v \ln \theta +A(\theta) q^2 \right)\right)_t+\left( r^2 \rho u \left( C_v \ln \theta +A(\theta) q^2 \right)\right)_r+R \rho (r^2 u)_r+\left( r^2 \frac{q}{\theta}\right)_r\\
&=\frac{1}{\kappa(\theta) \theta^2} \left(r^2+\tau(\theta) \rho u 2 r \right) q^2+ \frac{\mu(\theta)}{\theta} r^2 \left( 2u_r^2+\frac{4}{r^2}u^2-\frac{2}{3} \left(u_r+\frac{2}{r} u\right)^2\right)+ \frac{\lambda(\theta)}{\theta} r^2 \left(u_r+\frac{2}{r} u\right)^2.
\end{align*}}
Then, using the mass equation,
\[
(r^2 \rho)_t+(r^2 \rho u)_r=0,
\]
and the energy equation,
\begin{align*}
\left( r^2 \rho (C_v \theta+a(\theta) q^2+\frac{1}{2}u^2) \right)_t +\left( r^2 \rho u \left( C_v\theta+a(\theta) q^2+\frac{1}{2}u^2\right)+r^2(pu+q)-\right.\\
\left.r^2 u \left( \frac{4}{3} \mu(\theta) (u_r-\frac{u}{r})+\lambda(\theta)(u_r+\frac{2}{r} u) \right)\right)_r=0,
\end{align*}
we conclude that{\small
\begin{align}
\left( r^2 \rho  C_v(\theta-\ln \theta-1)+r^2 R(\rho \ln \rho-\rho+1)+r^2 \rho (a(\theta)-A(\theta) )q^2+\frac{1}{2}r^2 \rho  u^2 \right)_t \nonumber \\
+\left( r^2 \rho u \left( C_v(\theta- \ln \theta-1)+R\ln \rho-1+\frac{1}{2}u^2+(a(\theta)-A(\theta) )q^2\right)+r^2(pu+q)-r^2\frac{q}{\theta}-\right. \nonumber\\
\left.r^2 u \left( \frac{4}{3} \mu(\theta) (u_r-\frac{u}{r})+\lambda(\theta)(u_r+\frac{2}{r} u) \right)\right)_r
+\frac{1}{\kappa(\theta) \theta^2} \left(r^2+\tau(\theta) \rho u 2 r \right) q^2 \nonumber\\
+ \frac{\mu(\theta)}{\theta} r^2 \left( 2u_r^2+\frac{4}{r^2}u^2-\frac{2}{3} \left(u_r+\frac{2}{r} u\right)^2\right)+\frac{\lambda (\theta)}{\theta} r^2 \left(u_r+\frac{2}{r} u\right)^2 =0. \label{new4}
\end{align}}
A Taylor expansion together with \eqref{new1} imply
\begin{align}
C_0(\rho-1)^2 \le \rho \ln \rho -\rho+1 \le C_1 (\rho-1)^2,\\
C_0 (\theta-1)^2 \le \theta-\ln \theta-1 \le C_1 (\theta-1)^2,
\end{align}
where $C_0, C_1$ are two positive constants. Furthermore, since $|u|_{L^\infty} \le C E^\frac{1}{2}$, we can choose $\delta$  (in Proposition \ref{pro})
small enough such that
\[
r^2+\tau(\theta) \rho u 2 r \ge \frac{1}{2} r^2.
\]
Therefore, integrating the equation \eqref{new4} over  $\Omega \times (0,t)$ and noting \eqref{new3}, we get the desired result  \eqref{1.7} in Lemma \ref{le1}.
\end{proof}

\begin{lemma}\label{le2}
There exists some constant $C$ such that
\begin{align}\label{1.9}
 \int_{\Omega} \left( r^2 \left((u_t)^2+(\rho_t)^2+(\theta_t)^2+\tau (q_t)^2 \right)\right)  \dif r +\int_0^t \int_{\Omega} r^2 ((q_t)^2+(\frac{4}{3} \mu+\lambda) u_{tr}^2) \dif r \dif t \nonumber\\
 \le C(E_0+ E^\frac{1}{2}(t)\int_0^t  \mathcal D(s)\dif s).
\end{align}
\end{lemma}
 \begin{proof}
 Taking the $t$-derivative of \eqref{1.5}, we have
 \begin{align}\label{1.10}
 \begin{cases}
 \rho_{tt}+\rho u_{rt}+u \rho_{rt}+\frac{2}{r} \rho u_t=F_1,\\
 \rho u_{tt}+\rho u u_{rt}+R\theta \rho_{rt}+R\rho \theta_{rt}-(\frac{4}{3} \mu(\theta)+\lambda(\theta))\left(u_{rr}+\frac{2}{r} u_r-\frac{2}{r^2} u\right)_t=F_2,\\
 \rho e_\theta \theta_{tt}+B(\rho, u, \theta, q) \theta_{rt}+p \left(u_r+\frac{2}{r} u\right)_t+\left( q_r+\frac{2}{r}q\right)_t=F_3,\\
 \tau(\theta) \rho\left( q_{tt}+u q_{rt}+u \left(\frac{2}{r} q\right)_t\right)+q_t+\kappa(\theta) \theta_{rt}=F_4,
 \end{cases}
 \end{align}
 where   we denote $B:=\rho u e_\theta-\frac{2a(\theta)}{Z(\theta)} q$ and
 \begin{align*}
 F_1:&=-\rho_t u_r-u_t \rho_r-\frac{2}{r} \rho_t u,\\
 F_2:&=-\rho_t u_t-(\rho u)_t u_r-R \theta_t \rho_r-R \rho_t \theta_r \\
 +&\left(\frac{4}{3} \mu(\theta)+\lambda(\theta)\right)_t\left(u_{rr}+\frac{2}{r} u_r-\frac{2}{r^2} u\right)+\left(\lambda^\prime(\theta)\theta_r\left(u_r+\frac{2}{r}u\right)+\frac{4}{3}\mu^\prime(\theta) \theta_r\left(u_r-\frac{u}{r}\right)\right)_t,\\
 F_3:&= -(\rho e_\theta)_t \theta_t-B_t \theta_r-\rho_t\left(u_r+\frac{2}{r} u\right)+\left( \left(\frac{2}{\tau(\theta)}+\frac{4}{r} u\right)a(\theta) q^2 \right)_t\\
 &\qquad\qquad\qquad+\left(\mu(\theta)\left( 2u_r^2+\frac{4}{r^2}u^2-\frac{2}{3} \left(u_r+\frac{2}{r} u\right)^2\right)+\lambda(\theta)\left(u_r+\frac{2}{r} u\right)^2\right)_t,\\
 F_4:&=-(\tau(\theta) \rho)_t q_t -(\tau(\theta) \rho u)_t q_r-(\tau(\theta) \rho u)_t \frac{2}{r} q-\kappa(\theta)_t\theta_r.
 \end{align*}
 Multiplying $\eqref{1.10}_2$ by $r^2u_t$, integrating with respect to $r$, we get
 \begin{align} \label{1.11}
 \int_{\Omega}  (\rho u_{tt}&+\rho u u_{rt}) r^2 u_t \dif r+ \int_{\Omega} R \theta \rho_{rt} r^2 u_t  \dif r +\int_{\Omega} R \rho \theta_{rt} r^2 u_t \dif r\nonumber\\
&-\int_{\Omega} (\frac{4}{3} \mu(\theta)+\lambda(\theta))\left(u_{rr}+\frac{2}{r} u_r-\frac{2}{r^2} u\right)_t r^2 u_t=\int_{\Omega} F_2 r^2 u_t \dif r.
 \end{align}
 We estimate each term as follows.

 First, the mass equation $\eqref{1.5}_1$ gives
 \begin{align} \label{3.15a}
 \int_{\Omega}  (\rho u_{tt}+\rho u u_{rt}) r^2 u_t \dif r=\int_{\Omega} \rho r^2 \left(\frac{1}{2} u_t^2 \right)_t +r^2 \rho u \left(\frac{1}{2} u_t^2\right)_r \dif r
 =\frac{\dif}{\dif t} \int_{\Omega} \frac{1}{2} \rho r^2 u_t^2 \dif r.
 \end{align}
Second, using again the boundary condition $u_t|_{\partial \Omega}=0$, one has
\begin{align}
\int_{\Omega} R \theta r^2 \rho_{rt} u_t \dif r&= -\int_{\Omega} \rho_t (R \theta r^2 u_t)_r \dif r
=-\int_{\Omega} R\rho_t \theta_r r^2 u_t \dif r -\int_{\Omega} R \rho_t \theta (r^2 u_t)_r \dif r \nonumber\\
& \ge  -CE^\frac{1}{2}(t) \mathcal D(t)-\int_{\Omega} R\rho_t \theta \left( -\frac{1}{\rho} r^2 \rho_{tt}-\frac{1}{\rho} r^2 u \rho_{tr}+\frac{1}{\rho} r^2 F_1\right)\dif r \nonumber\\
&\ge -CE^\frac{1}{2}(t) \mathcal D(t)+\frac{\dif}{\dif t} \int_{\Omega} \frac{R\theta}{2\rho} r^2 \rho_t^2 \dif r.
\label{3.15b}
\end{align}
Similarly, one has
\begin{align}
&-\int_{\Omega} (\frac{4}{3} \mu(\theta)+\lambda(\theta))\left(u_{rr}+\frac{2}{r} u_r-\frac{2}{r^2} u\right)_t r^2 u_t \nonumber\\&
=-\int_{\Omega} (\frac{4}{3} \mu(\theta)+\lambda(\theta))\left( (r^2 u_r)_{rt}-2 u_t\right) u_t \dif r \nonumber \\
&\ge \int_{\Omega}  (\frac{4}{3} \mu(\theta)+\lambda(\theta))(r^2 u_{rt}^2+2u_t^2)\dif r
  -CE^\frac{1}{2}(t) \mathcal D(t).
\label{3.15c}
\end{align}
So, we have from \eqref{1.11}-\eqref{3.15c}
\begin{align}\label{1.12}
\frac{\dif}{\dif t} \int_{\Omega} \left(\frac{1}{2} \rho r^2 u_t^2+\frac{\theta}{2\rho} r^2 \rho_t^2 \right)\dif r + \int_{\Omega}  R \rho \theta_{rt} r^2 u_t \dif r +&\int_{\Omega}  (\frac{4}{3} \mu(\theta)+\lambda(\theta))(r^2 u_{rt}^2+2u_t^2)\dif r
\nonumber\\
&\le CE^\frac{1}{2}(t) \mathcal D(t).
 \end{align}

 Multiplying $\eqref{1.10}_3$ by $\frac{1}{\theta} r^2 \theta_t$, we have
 \begin{align}
 \int_{\Omega} \frac{\rho e_\theta}{\theta} r^2 \theta_{tt}\theta_t \dif r+ \int_{\Omega} \frac{1}{\theta} r^2  B \theta_{rt} \theta_t \dif r +\int_{\Omega}  R \rho r^2(u_r+\frac{2}{r}u)_t  \theta_t \dif r
 +&\int_{\Omega} \frac{1}{\theta} r^2 (q_r+\frac{2}{r} q)_t \theta_t \dif r  \nonumber\\
 &=\int_{\Omega} \frac{1}{\theta} r^2 F_3  \theta_t \dif r . \label{1.13}
 \end{align}

  We estimate each term as follows.
First,
\begin{align}
\int_{\Omega} r^2 \frac{\rho e_\theta}{\theta} \theta_{tt} \theta_t \dif r& =\frac{\dif}{\dif t} \int_{\Omega} \frac{\rho e_\theta}{2\theta} r^2 \theta_t^2 \dif r -\int_{\Omega} \left( \frac{\rho e_\theta}{2\theta}\right)_t r^2 \theta_t^2 \dif r \nonumber\\
&\ge \frac{\dif}{\dif t} \int_{\Omega} \frac{\rho e_\theta}{2 \theta} r^2 \theta_t^2 \dif r- CE^\frac{1}{2}(t) \mathcal D(t). \label{3.17a}
\end{align}
Second, using the fact $B|_{\partial \Omega}=0$, one has
\begin{align}
\int_{\Omega} \frac{1}{\theta} r^2 B \theta_{rt} \theta_t \dif r = \int_{\Omega} \frac{1}{\theta} r^2 B \left(\frac{1}{2} \theta_t^2\right)_r \dif r
=-\int_{\Omega} \left(\frac{1}{\theta} r^2 B\right)_r \frac{1}{2}\theta_t^2 \dif r \ge -CE^\frac{1}{2}(t) \mathcal D(t). \label{3.17b}
\end{align}
 Third, using the fact $u_t|_{\partial \Omega}=0$, we get
 \begin{align}
 \int_{\Omega} R \rho r^2 (u_r+\frac{2}{r}u)_t \theta_t \dif r&=\int_{\Omega} R \rho (r^2 u)_{rt} \theta_t \dif r=-\int_{\Omega} R(r^2 u)_t  (\rho_r \theta_t+\rho \theta_{tr})\dif r \nonumber\\
& \ge -\int_{\Omega} R \rho r^2 u_t \theta_{tr} \dif r- CE^\frac{1}{2}(t) \mathcal D(t).
\label{3.17c}
 \end{align}
 The last term on the left-hand side of equation \eqref{1.13} reduces to
 \begin{align}
 \int_{\Omega} \frac{1}{\theta} r^2 \left( q_r+\frac{2}{r} q\right)_t \theta_t \dif r =\int_{\Omega} \frac{1}{\theta} (r^2 q)_{rt} \theta_t \dif r. \label{3.17d}
 \end{align}
 Thus, we derive from \eqref{1.13}-\eqref{3.17d}
 \begin{align}\label{1.14}
 \frac{\dif}{\dif t} \int_{\Omega} \frac{\rho e_\theta}{2\theta} r^2 \theta_t^2 \dif r -\int_{\Omega} R \rho r^2 u_t \theta_{tr} \dif r+\int_{\Omega} \frac{1}{\theta} (r^2 q)_{rt} \theta_t \dif r \le CE^\frac{1}{2}(t) \mathcal D(t).
 \end{align}

 Multiplying $\eqref{1.10}_4$ by $\frac{1}{\kappa(\theta) \theta} r^2 q_t$, one gets
 \begin{align}
 \int_{\Omega}  \frac{\tau(\theta) }{\kappa(\theta) \theta} \rho r^2 q_{tt}  q_t\dif r&+\int_{\Omega} \frac{\tau(\theta) }{\kappa(\theta) \theta} \rho u r^2 q_{rt} q_t \dif r +\int_{\Omega} \frac{\tau(\theta) }{\kappa(\theta) \theta} \rho u r^2 \left(\frac{2}{r}q \right)_t q_t\dif r \nonumber\\
 &+\int_{\Omega} \frac{1}{\kappa(\theta)\theta} r^2 q_t^2 \dif r+\int_{\Omega} \frac{1}{\theta} r^2 q_t \theta_{rt} \dif r
 =\int_{\Omega} \frac{1}{\kappa(\theta)\theta} r^2 F_4 q_t \dif r. \label{1.15}
 \end{align}

 We estimate each term as follows.

 First,
 \begin{align*}
 \int_{\Omega} \frac{\tau(\theta) }{\kappa(\theta) \theta} \rho r^2 q_{tt} q_t\dif r&=\frac{\dif}{\dif t} \int_{\Omega} \frac{\tau(\theta) }{\kappa(\theta) \theta}\rho r^2 \frac{1}{2} q_t^2\dif r -\int_{\Omega} \left(\frac{\tau(\theta) }{\kappa(\theta) \theta} \rho r^2 \right)_t \frac{1}{2} q_t^2 \dif r\\
 & \ge \frac{\dif}{\dif t} \int_{\Omega} \frac{\tau(\theta) }{\kappa(\theta) \theta}\rho r^2 \frac{1}{2} q_t^2\dif r-CE^\frac{1}{2}(t) \mathcal D(t).
 \end{align*}
 Second, using the boundary condition $u|_{\partial \Omega}=0$, one has
 \begin{align*}
 \int_{\Omega} \frac{\tau(\theta) }{\kappa(\theta) \theta}  \rho u r^2 q_{tr} q_t \dif r &= \int_{\Omega} \frac{\tau(\theta) }{\kappa(\theta) \theta}  \rho u r^2 \left(\frac{1}{2} q_t^2\right)_r \dif r\\
 &=-\int_{\Omega} \left(\frac{\tau(\theta) }{\kappa(\theta) \theta}  \rho u r^2 \right)_r \frac{1}{2} q_t^2 \dif r \ge -CE^\frac{1}{2}(t) \mathcal D(t).
 \end{align*}
 Third,
 {\small
 \begin{align*}
 \int_{\Omega} \frac{\tau(\theta) }{\kappa(\theta) \theta}  \rho u r^2 \left(\frac{2}{r} q\right)_t q_t +\frac{1}{\kappa(\theta)\theta} r^2 q_t^2 \dif r
 =\int_{\Omega} \left( \tau(\theta) \rho u \frac{2}{r}+1\right) \frac{1}{\kappa(\theta)\theta} r^2 q_t^2 \dif r \ge \int_{\Omega} \frac{1}{2\kappa(\theta)\theta} r^2 q_t^2\dif r.
 \end{align*}}
 For the last term on the left hand side of equation \eqref{1.15}, one has
 \begin{align*}
 \int_{\Omega} \frac{1}{\theta} r^2 q_t \theta_{rt} \dif r =\int_{\Omega} \left(\frac{1}{\theta}\right)_r r^2 q_t \theta_t \dif r-\int_{\Omega} \frac{1}{\theta} (r^2 q_t)_r \theta_t \dif r \ge -\int_{\Omega} \frac{1}{\theta} (r^2 q)_{rt} \theta_t \dif r -CE^\frac{1}{2}(t) \mathcal D(t).
 \end{align*}
 Thus, we derive that
 \begin{align} \label{1.16}
 \frac{\dif}{\dif t} \int_{\Omega} \frac{\rho \tau(\theta) }{2\kappa(\theta) \theta}   r^2 q_t^2 \dif r -\int_{\Omega} \frac{1}{\theta} (r^2 q)_{rt} \theta_t \dif r +\int_{\Omega} \frac{1}{2\kappa(\theta)\theta} r^2 q_t^2 \dif r \le CE^\frac{1}{2}(t) \mathcal D(t).
 \end{align}
 Therefore, combining \eqref{1.12}, \eqref{1.14} and \eqref{1.16}, we have
 \begin{align}\label{1.17}
 \frac{\dif }{\dif t} \int_{\Omega} \left( \frac{\rho}{2}  r^2 u_t^2+\frac{R\theta}{2\rho} r^2 \rho_t^2+\frac{\rho e_\theta}{2\theta} r^2 \theta_t^2 +\frac{\rho\tau(\theta)}{2\kappa(\theta)\theta} r^2 q_t^2 \right) \dif r
 +\int_{\Omega} \frac{1}{2\kappa(\theta) \theta} r^2 q_t^2 \dif r
 \nonumber\\
+\int_{\Omega}  (\frac{4}{3} \mu(\theta)+\lambda(\theta))(r^2 u_{rt}^2+2u_t^2)\dif r
 \le CE^\frac{1}{2}(t) \mathcal D(t).
 \end{align}
Integrating over $(0,t)$,  and noticing the fact that
 \[
 \int_{\Omega} \left( \frac{\rho}{2}  r^2 u_t^2+\frac{R\theta}{2\rho} r^2 \rho_t^2+\frac{\rho e_\theta}{2\theta} r^2 \theta_t^2 +\frac{\rho\tau(\theta)}{2\kappa(\theta)\theta} r^2 q_t^2 \right) (t=0,r)\dif r\le   CE_0,
 \]
 we get the desired result \eqref{1.9} in Lemma \ref{le2}.
 \end{proof}

\begin{lemma}\label{le3}
There exists some constant $C$ such that
\begin{align}\label{1.18}
 \int_{\Omega} \left( \left(r^2 u_r^2+2u^2\right)+r^2 \rho_r^2+r^2 \theta_r^2+\tau (r^2 q_r^2+2q^2) \right)  \dif r +\int_0^t \int_{\Omega} (r^2 q_r^2+2q^2) \dif r \dif t  \nonumber\\
  \le C(E_0+ E^\frac{1}{2}(t) \int_0^t \mathcal D(s)\dif s ).
\end{align}
\end{lemma}

\begin{proof}
 Taking the $r$-derivative of \eqref{1.5}, we get
 \begin{align}\label{1.19}
 \begin{cases}
 \rho_{tr}+u\rho_{rr}+\rho u_{rr}+\frac{2}{r} \rho u_r-\frac{2}{r^2} \rho u =G_1,\\
 \rho u_{tr}+\rho u u_{rr}+p_{rr}-\left((\frac{4}{3} \mu(\theta)+\lambda(\theta))\left(u_{rr}+\frac{2}{r} u_r-\frac{2}{r^2} u\right)\right)_r=G_2,\\
 \rho e_\theta \theta_{tr}+ B(\rho, u, \theta, q) \theta_{rr}+p\left(u_r+\frac{2}{r} u\right)_r+\left(q_r+\frac{2}{r}q\right)_r=G_3,\\
 \tau(\theta) \rho (q_{tr}+ u q_{rr})+q_r+\kappa(\theta) \theta_{rr} =G_4,
 \end{cases}
 \end{align}
 where
 \begin{align*}
 G_1:&=-2\rho_r u_r-\frac{2}{r} \rho_r u,\\
 G_2:&=-\rho_r u_t-(\rho u)_r u_r+ \left(\lambda^\prime(\theta)\theta_r\left(u_r+\frac{2}{r}u\right)+\frac{4}{3}\mu^\prime(\theta) \theta_r\left(u_r-\frac{u}{r}\right)\right)_r,\\
 G_3:&= -(\rho e_\theta)_r \theta_t -B_r(\rho ,u, \theta, q) \theta_r-p_r(u_r+\frac{2}{r} u)+\left( \left(\frac{2}{\tau(\theta)}+\frac{4}{r} u\right) a(\theta) q^2 \right)_r\\
&\qquad  +\left(\mu(\theta)\left( 2u_r^2+\frac{4}{r^2}u^2-\frac{2}{3} \left(u_r+\frac{2}{r} u\right)^2\right)+\lambda(\theta)\left(u_r+\frac{2}{r} u\right)^2\right)_r\\
 G_4:&=-(\tau(\theta) \rho)_r q_t-(\tau(\theta) \rho u)_r q_r- \left( \tau(\theta) \rho u \frac{2}{r} q\right)_r-\kappa^\prime(\theta) \theta_r^2.
 \end{align*}
 Multiplying $\eqref{1.19}_2$ by $r^2 u_r$, we get
 \begin{align}
 \int_{\Omega} \rho u_{tr} r^2 u_r \dif r+ \int_{\Omega} \rho u u_{rr} r^2 u_r \dif r+\int_{\Omega} p_{rr} r^2 u_r\dif r=\int_{\Omega} G_2 r^2 u_r \dif r.
 \end{align}

 We estimate the different terms as follows.

 First, one has
 \begin{align*}
 \int_{\Omega} \left( \rho u_{tr} r^2 u_r+\rho u u_{rr} r^2 u_r\right)\dif r =\int_{\Omega} \rho r^2 \left(\frac{1}{2} u_r^2\right)_t +\rho u r^2 \left(\frac{1}{2} u_r^2 \right)_r \dif r
 =\frac{\dif}{\dif r} \int_{\Omega} \frac{1}{2} \rho r^2 u_r^2 \dif r.
 \end{align*}
 Second, since $(u, u_t, \theta_r)|_{\partial \Omega}=0$, we derive from the momentum equation \eqref{1.5} that
 $$
 \left(p_r-(\frac{4}{3} \mu(\theta)+\lambda(\theta))\left(u_{rr}+\frac{2}{r} u_r-\frac{2}{r^2} u\right)\right)\Big|_{\partial\Omega}=0,
 $$
  so one gets
 \begin{align*}
& \int_{\Omega} \left(p_{rr}- \left((\frac{4}{3} \mu(\theta)+\lambda(\theta))\left(u_{rr}+\frac{2}{r} u_r-\frac{2}{r^2} u\right)\right)_r \right)r^2 u_r \dif r\\
&  =-\int_{\Omega}\left( R(\theta \rho_r+\rho \theta_r)- (\frac{4}{3} \mu(\theta)+\lambda(\theta))\left(u_{rr}+\frac{2}{r} u_r-\frac{2}{r^2} u\right)\right) (r^2 u_r)_r \dif r.
 \end{align*}
 Using the following equation, arising from the mass equation,
 \begin{align}
 r^2 \rho_{tr}+2r^2 \rho_r u_r+r^2 u \rho_{rr} +\rho (r^2 u_r)_r+2 r \rho_r u-2\rho u=0,
 \end{align}
 we have
 \begin{align*}
 &-\int_{\Omega} R \theta \rho_r (r^2 u_r)_r \dif r=-\int_{\Omega} R \theta \rho_r( 2r u_r+r^2 u_{rr}) \dif r \\
 &=\int_{\Omega} R\theta \rho_r \frac{1}{\rho} \left( r^2 \rho_{tr}+2r^2 \rho_r u_r+r^2 u \rho_{rr} +2 r \rho_r u- 2\rho u\right) \dif r\\
 &=\int_{\Omega}  \frac{R\theta}{\rho} r^2 \left(\frac{1}{2} \rho_r^2\right)_t\dif r+ \int_{\Omega} \frac{R\theta}{\rho} 2 r^2 \rho_r^2 u_r \dif r+\int_{\Omega} \frac{R\theta}{\rho} \rho_r r^2 u \rho_{rr}\dif r
 +\int_{\Omega} \frac{R\theta}{\rho} \rho_r (2 r \rho_r u-2\rho u) \dif r\\
& \ge \frac{\dif}{\dif t} \int_{\Omega} \frac{R\theta}{2\rho} r^2  \rho_r^2 \dif r-\int_{\Omega} 2 R\theta \rho_r u \dif r-C E^\frac{1}{2}(t) \mathcal D(t).
 \end{align*}
 Moreover, we have
 \begin{align*}
& \int_{\Omega} \left(\frac{4}{3} \mu(\theta)+\lambda(\theta)\right)\left(u_{rr}+\frac{2}{r}u_r-\frac{2}{r^2} u\right) (r^2 u_r)_r \dif r\\
  &=\int_{\Omega} \left(\frac{4}{3} \mu(\theta)+\lambda(\theta)\right)\left(u_{rr}+\frac{2}{r}u_r-\frac{2}{r^2} u\right) (r^2 u_{rr}+2r u_r) \dif r\\
  &=\int_{\Omega} \left(\frac{4}{3} \mu(\theta)+\lambda(\theta)\right) \left( r^2 u_{rr}^2+4u_r^2+4r u_{rr} u_r -2 u u_{rr}-\frac{4}{r} u u_r\right)\dif r\\
  &\ge   \int_{\Omega} \left(\frac{4}{3} \mu(\theta)+\lambda(\theta)\right)\left(\frac{1}{5} r^2 u_{rr}^2+u_r^2\right) \dif r-   \int_{\Omega} \left(\frac{4}{3} \mu(\theta)+\lambda(\theta)\right)\frac{2}{r^2} u^2-C E^\frac{1}{2}(t) \mathcal D(t).
  \end{align*}
 Thus, we derive
 \begin{align}\label{1.20}
 \frac{\dif}{\dif t} \int_{\Omega} \left( \frac{1}{2} \rho r^2 u_r^2 +\frac{\theta}{2\rho} r^2 \rho_r^2\right) \dif r -\int_{\Omega} R \rho \theta_r (r^2 u_r)_r\dif r-\int_{\Omega} 2 R \theta \rho_r u \dif r  \nonumber\\
  +\int_{\Omega} \left(\frac{4}{3} \mu(\theta)+\lambda(\theta)\right)\left(\frac{1}{5} r^2 u_{rr}^2+u_r^2\right) \dif r-   \int_{\Omega} \left(\frac{4}{3} \mu(\theta)+\lambda(\theta)\right)\frac{2}{r^2} u^2\le C E^\frac{1}{2}(t) \mathcal D(t).
 \end{align}

 Multiplying $\eqref{1.19}_3$ by $\frac{1}{\theta} r^2 \theta_r$, one has
 \begin{align}
 \int_{\Omega} \rho e_\theta \theta_{tr}  \frac{1}{\theta} r^2 \theta_r\dif r+ \int_{\Omega} B(\rho, u, \theta, q) \theta_{rr} \frac{1}{\theta} r^2 \theta_r \dif r+\int_{\Omega} R\rho r^2\left(u_r+\frac{2}{r} u\right)_r \theta_r \dif r \nonumber\\
 +\int_{\Omega} \left( q_r+\frac{2}{r} q\right)_r \frac{1}{\theta} r^2 \theta_r \dif r=\int_{\Omega} G_3 \frac{1}{\theta} r^2 \theta_r \dif r  .\label{1.21}
 \end{align}

 We estimate each term in the above equation as follows.

 First, one has
 \begin{align*}
 \int_{\Omega} \rho e_\theta \theta_{tr} \frac{1}{\theta} r^2 \theta_r \dif r=\int_{\Omega} \frac{\rho e_\theta}{\theta} r^2 \left(\frac{1}{2} \theta_r^2\right)_t \dif r \ge \frac{\dif}{\dif t} \int_{\Omega} \frac{\rho e_\theta}{2\theta} r^2 \theta_r^2 \dif r
 -C E^\frac{1}{2}(t) \mathcal D(t).
 \end{align*}
 Second, using the fact $B(\rho, u, \theta, q)|_{\partial \Omega}=0$, we have
 \begin{align*}
 \int_{\Omega} B(\rho, u, \theta, q) \theta_{rr} \frac{1}{\theta} r^2\theta_r \dif r=-\int_{\Omega} \left(\frac{B(\rho,u, \theta, q)}{\theta}r^2 \right)_r \frac{1}{2} \theta_r^2 \dif r\ge -C E^\frac{1}{2}(t) \mathcal D(t).
 \end{align*}
 Third,
 \begin{align*}
 \int_{\Omega} R \rho r^2 \left( u_r+\frac{2}{r}u \right)_r \theta_r \dif r =\int_{\Omega} R \rho r^2 (u_{rr}+\frac{2}{r}u_r-\frac{2}{r^2} u)\theta_r \dif r\\
 =\int_{\Omega}  R\rho ( (r^2 u_r)_r-2u)\theta_r\dif r= \int_{\Omega} R\rho (r^2 u_r)_r\theta_r \dif r-\int_{\Omega} 2R \rho u \theta_r \dif r.
 \end{align*}
 Similarly, one has
 \begin{align*}
 \int_{\Omega} \left( q_r+\frac{2}{r} q\right)_r \frac{1}{\theta} r^2 \theta_r \dif r= \int_{\Omega} \frac{1}{\theta} (r^2 q_r)_r \theta_r \dif r- \int_{\Omega} \frac{2}{\theta} q \theta_r\dif r.
 \end{align*}
 So, we derive that
 \begin{align}\label{1.22}
 \frac{\dif}{\dif t} \int_{\Omega} \frac{\rho e_\theta}{2\theta} r^2 \theta_r^2 \dif r +\int_{\Omega}  \left(R \rho(r^2 u_r)_r \theta_r+\frac{1}{\theta} (r^2 q_r)_r \theta_r\right) \dif r
 -\int_{\Omega} (2 R \rho u \theta_r + \frac{2}{\theta} q \theta_r )\dif r \nonumber\\
 \le  CE^\frac{1}{2}(t) \mathcal D(t).
 \end{align}

 Multiplying $\eqref{1.19}_4$ by $\frac{1}{\kappa(\theta)\theta} r^2 q_r$, one get
 \begin{align} \label{1.23}
 \int_{\Omega} \frac{\tau(\theta)}{\kappa(\theta)\theta} \rho r^2 ( q_{tr}+u q_{rr}) q_r\dif r+\int_{\Omega} \frac{1}{\kappa(\theta)\theta} r^2 q_r^2 \dif r+\int_{\Omega} \frac{1}{\theta} \theta_{rr} (r^2 q_r) \dif r=\int_{\Omega} \frac{1}{\kappa(\theta)\theta} r^2 q_r G_4 \dif r.
 \end{align}
 We estimate each term in the above equation as follows.

 First, using the mass equation, one has
 \begin{align*}
  \int_{\Omega} \frac{\tau(\theta)}{\kappa(\theta)\theta} \rho r^2 ( q_{tr}+u q_{rr}) q_r\dif r&= \int_{\Omega} \frac{\tau(\theta)}{\kappa(\theta)\theta} \rho r^2 ( (\frac{1}{2} q_r^2)_t+u (\frac{1}{2}q_r^2)_r) \dif r\\
  &\ge \frac{\dif}{\dif t} \int_{\Omega} \frac{\tau(\theta)}{\kappa(\theta)\theta} \rho r^2  \frac{1}{2} q_r^2 \dif r -CE^\frac{1}{2}(t) \mathcal D(t).
  \end{align*}
 Second, use the fact $\theta_r|_{\partial \Omega}=0$, we have
 \begin{align*}
 \int_{\Omega} \frac{1}{\theta} \theta_{rr} r^2 q_r \dif r \ge -\int_{\Omega} \frac{1}{\theta} \theta_r (r^2 q_r)_r\dif r -CE^\frac{1}{2}(t) \mathcal D(t).
 \end{align*}
 Thus, we derive that
 \begin{align}\label{1.24}
 \frac{\dif}{\dif t} \int_{\Omega} \frac{\tau(\theta)}{2\kappa(\theta)\theta} \rho r^2 q_r^2 \dif r -\int_{\Omega} \frac{1}{\theta} \theta_r (r^2 q_r)_r \dif r+\int_{\Omega} \frac{1}{\kappa(\theta)\theta} r^2 q_r^2 \dif r \le CE^\frac{1}{2}(t) \mathcal D(t).
 \end{align}
 Combining \eqref{1.20}, \eqref{1.22} and \eqref{1.24}, we obtain{\small
 \begin{align}
& \frac{\dif}{\dif t} \int_{\Omega} \left( \frac{1}{2}\rho r^2 u_r^2+\frac{\theta}{2\rho} r^2 \rho_r^2+\frac{\rho e_\theta}{2\theta} r^2 \theta_r^2+\frac{\tau(\theta)}{2\kappa(\theta)\theta} \rho r^2 q_r^2 \right)\dif r
+\int_{\Omega} \frac{1}{\kappa(\theta)\theta} r^2 q_r^2 \dif r -\int_{\Omega}  \left(2 u p_r+\frac{2}{\theta} q\theta_r\right)\dif r \nonumber \\
&+ \int_{\Omega} \left(\frac{4}{3} \mu(\theta)+\lambda(\theta)\right)\left(\frac{1}{5} r^2 u_{rr}^2+u_r^2\right) \dif r-   \int_{\Omega} \left(\frac{4}{3} \mu(\theta)+\lambda(\theta)\right)\frac{2}{r^2} u^2 \le CE^\frac{1}{2}(t) \mathcal D(t). \label{1.25}
 \end{align}}
 On the other hand, multiplying $\eqref{1.5}_2$ by $2u$, one gets{\small
 \begin{align*}
& \int_{\Omega} 2 \rho u_t  u \dif r +\int_{\Omega} 2\rho u u_r  u \dif r+\int_{\Omega}  2 p_r  u \dif r\\
&= \int_{\Omega} \left(\left(\frac{4}{3} \mu(\theta)+\lambda(\theta)\right)\left(u_{rr}+\frac{2}{r}u_r-\frac{2}{r^2} u\right)+\lambda^\prime(\theta)\theta_r\left(u_r+\frac{2}{r}u\right)+\frac{4}{3}\mu^\prime(\theta) \theta_r\left(u_r-\frac{u}{r}\right) \right)2 u \dif r ,
\end{align*}}
For the term on the right hand side of the above equation, we have{\small
\begin{align*}
&\int_{\Omega} \left(\left(\frac{4}{3} \mu(\theta)+\lambda(\theta)\right)\left(u_{rr}+\frac{2}{r}u_r-\frac{2}{r^2} u\right)+\lambda^\prime(\theta)\theta_r\left(u_r+\frac{2}{r}u\right)+\frac{4}{3}\mu^\prime(\theta) \theta_r\left(u_r-\frac{u}{r}\right) \right)2 u \dif r\\
&\le -\int_{\Omega} \left(\frac{4}{3} \mu(\theta)+\lambda(\theta)\right)\left(2u_r^2+\frac{2}{r^2} u^2 \right)\dif r+CE^\frac{1}{2}(t) \mathcal D(t).
\end{align*}}
Thus, we derive
\begin{align}\label{1.26}
\frac{\dif}{\dif t} \int_{\Omega} \rho u^2 \dif r+\int_{\Omega}  2p_r  u \dif r+\int_{\Omega} \left(\frac{4}{3} \mu(\theta)+\lambda(\theta)\right)\left(2u_r^2+\frac{2}{r^2} u^2 \right)\dif r \le CE^\frac{1}{2}(t) \mathcal D(t).
 \end{align}
 Then, multiplying  $\eqref{1.5}_4$ by $\frac{2}{\kappa(\theta)\theta} q$, one has
 \begin{align*}
 \int_{\Omega} \tau(\theta) \rho (q_t+uq_r+u\cdot \frac{2}{r} q) \frac{2}{\kappa (\theta)\theta} q \dif r+\int_{\Omega} \frac{2}{\kappa(\theta)\theta} q^2 \dif r+\int_{\Omega} \frac{2}{\theta} q \theta_r \dif r=0,
 \end{align*}
 from which we get
 \begin{align}\label{1.27}
 \frac{\dif}{\dif t} \int_{\Omega} \frac{  \rho \tau(\theta)}{\kappa(\theta) \theta} q^2 \dif r +\int_{\Omega} \frac{2}{\kappa(\theta)\theta} q^2 \dif r+\int_{\Omega} \frac{2}{\theta} q \theta_r \dif r \le CE^\frac{1}{2}(t) \mathcal D(t).
 \end{align}

 Therefore, combining \eqref{1.25}, \eqref{1.26} and \eqref{1.27}, we conclude that
 \begin{align*}
  \frac{\dif}{\dif t} \int_{\Omega} \left(\frac{1}{2} \rho (r^2 u_r^2+2u^2)+\frac{\theta}{2\rho} r^2\rho_r^2+\frac{\rho e_\theta}{2\theta} r^2 \theta_r^2+\frac{\tau(\theta)}{2\kappa(\theta)\theta}(r^2 q_r^2+2q^2)\right)\dif r  \\
+  \int_{\Omega} \left(\frac{4}{3} \mu(\theta)+\lambda(\theta)\right)\left(\frac{1}{5} r^2 u_{rr}^2+3u_r^2\right) \dif r +\int_{\Omega} \frac{1}{\kappa(\theta)\theta} (r^2 q_r^2+2q^2) \dif r \le CE^\frac{1}{2}(t) \mathcal D(t).
 \end{align*}
 Integrating the above result with respect to $t$, the proof  of Lemma \ref{le3} is finished.
 \end{proof}

 Using the equation $\eqref{1.5}_3$ and $\eqref{1.5}_4$ and Lemmas \ref{le1}-\ref{le3}, we obtain
 \begin{lemma}\label{le4}
 There exists some constant $C$ such that
 \begin{align}\label{1.29}
 \int_{\Omega} r^2(q^2+q_r^2) \dif r \le C(E_0+E^\frac{1}{2}(t) \int_0^t \mathcal D(s) \dif s + E^2(t)).
 \end{align}
 \end{lemma}
 Now we continue with
 \begin{lemma}\label{le5}
  There exists some constant $C$ such that
\begin{align}\label{1.30}
\int_0^t \int_{\Omega}( r^2 |D(\rho, u, \theta, q)|^2 +u^2+q^2)\dif r \dif t \le C(E_0+E^\frac{1}{2}(t) \int_0^t  \mathcal D(s)\dif s).
\end{align}
\end{lemma}

\begin{proof}
Using equation $\eqref{1.5}_4$ and Lemmas \ref{le1}-\ref{le3}, we have
\begin{align}
\int_0^t \int_{\Omega} \kappa^2(\theta) r^2  \theta_r^2 \dif r \dif t &\le C \int_0^t \int_{\Omega} \left( r^2 q^2+\tau^2(\theta) \rho^2 r^2 q_t^2+\tau^2(\theta) \rho^2 r^2 q_r^2+\tau^2(\theta) \rho^2 u^2 4 q^2\right)\dif r \dif t  \nonumber\\
&\le C(E_0+E^\frac{1}{2}(t) \int_0^t  \mathcal D(s)\dif s). \label{1.31}
\end{align}

Multiplying $\eqref{1.5}_3$ by $\frac{1}{\rho e_\theta} (r^2u)_r$ and integrating the result, we get
\begin{align}
&\int_0^t \int_{\Omega} \frac{p}{\rho e_\theta} (u_r+\frac{2}{r} u) (r^2 u)_r \dif r \nonumber\\
=&-\intx  \theta_t (r^2 u)_r\dif r \dif t-\intx \frac{B(\rho, u, \theta, q)}{\rho e_\theta} \theta_r (r^2 u)_r \dif r \dif t-\intx \frac{1}{\rho e_\theta} (q_r+\frac{2}{r} q) (r^2 u)_r \dif r \dif t \nonumber\\
&+\intx \frac{1}{\rho e_\theta} \left(\frac{2}{\tau(\theta)}+\frac{4}{r} u\right) a(\theta) q^2 (r^2 u)_r \dif r \dif t \nonumber\\
&+ \intx \left(\mu(\theta)\left( 2u_r^2+\frac{4}{r^2}u^2-\frac{2}{3} \left(u_r+\frac{2}{r} u\right)^2\right)+\lambda(\theta)\left(u_r+\frac{2}{r} u\right)^2\right) (r^2 u)_r \dif r \dif t . \label{1.32}
\end{align}
We estimate each term in the above equation as follows.
First, one has
\begin{align*}
&\intx \frac{R\theta}{e_\theta} (u_r+\frac{2}{r} u)(r^2 u_r+2 r u) \dif r \dif t \\
&=\intx \frac{R\theta}{e_\theta} (r^2 u_r^2+4u^2+4 r u u_r)\dif r \dif t\\
&=\intx \frac{R\theta}{e_\theta} (r^2u_r^2+2u^2)\dif r \dif t-\intx \left(\frac{R\theta}{e_\theta}\right)_r r 2u^2 \dif r \dif t\\
&\ge \intx \frac{R}{C_v} (r^2 u_r^2+2u^2)\dif r \dif t - E^\frac{1}{2}(t) \int_0^t  \mathcal D(s)\dif s,
\end{align*}
where we use the fact that
\[
\|r(\rho_r, u_r, \theta_r, q_r)\|_{L^\infty}  \le CE^\frac{1}{2}(t).
\]
For the right hand side of \eqref{1.32},  using the boundary condition $u_t|_{\partial \Omega}=0$, we have
\begin{align*}
&-\intx \theta_t(r^2 u)_r \dif r \dif t =-\int_0^t \frac{\dif}{\dif t} \int_{\Omega} \theta (r^2 u)_r \dif r \dif t +\intx \theta(r^2 u)_{rt} \dif r \dif t \\
&\le -\intx \theta_r r^2 u_t \dif r \dif t +C(E_0+E^\frac{1}{2}(t) \int_0^t  \mathcal D(s)\dif s) \\
&\le \eta \intx r^2 u_t^2 \dif r \dif t +C(\eta) (E_0+E^\frac{1}{2}(t) \int_0^t  \mathcal D(s)\dif s).
\end{align*}
where $\eta$ is  a small constant to be determined later.

For the second term of the right hand side of \eqref{1.32}, one has
\begin{align*}
-\intx \frac{B}{\rho e_\theta} \theta_r (r^2 u)_r \dif r \dif t =-\intx \frac{B}{\rho e_\theta} \theta_r (r^2 u_r+2ru) \dif r \dif t \\
\le \frac{1}{2} \intx \frac{R}{C_v} (r^2 u_r^2+2u^2)\dif r \dif t +C(E_0+E^\frac{1}{2}(t) \int_0^t  \mathcal D(s)\dif s).
\end{align*}

Similarly, we have
\begin{align*}
&-\intx \frac{1}{\rho e_\theta} (q_r+\frac{2}{r}q) (r^2 u)_r \dif r \dif t \\
&\le \frac{1}{4} \intx \frac{R}{C_v} (r^2 u_r^2+2u^2)\dif r\dif t +C \intx (r^2 q_r^2+q^2)\dif r\dif t \\
&\le \frac{1}{4} \intx \frac{R}{C_v} (r^2 u_r^2+2u^2)\dif r\dif t +C(E_0+E^\frac{1}{2}(t) \int_0^t  \mathcal D(s)\dif s),
\end{align*}
and
\begin{align*}
&\intx \frac{1}{\rho e_\theta} \left( \frac{2}{\tau(\theta)}+\frac{4}{r} u\right) a(\theta) q^2 (r^2 u)_r \dif r \dif t \\
&\le \frac{1}{16} \intx \frac{R}{C_v} (r^2 u_r^2+2u^2)\dif r\dif t +C \intx (r^2 q^4)\dif r\dif t \\
&\le \frac{1}{16} \intx \frac{R}{C_v} (r^2 u_r^2+2u^2)\dif r\dif t +C(E_0+E^\frac{1}{2}(t) \int_0^t  \mathcal D(s)\dif s).
\end{align*}
In the same way, one has
\begin{align*}
&\intx \left(\mu(\theta)\left( 2u_r^2+\frac{4}{r^2}u^2-\frac{2}{3} \left(u_r+\frac{2}{r} u\right)^2\right)+\lambda(\theta)\left(u_r+\frac{2}{r} u\right)^2\right) (r^2 u)_r \dif r \dif t \\
&\le \frac{1}{16} \intx \frac{R}{C_v} (r^2 u_r^2+2u^2)\dif r\dif t +C(E_0+E^\frac{1}{2}(t) \int_0^t  \mathcal D(s)\dif s).
\end{align*}
Thus, we derive that
\begin{align}\label{1.33}
\frac{1}{8} \intx \frac{R}{C_v} (r^2 u_r^2+2u^2)\dif r\dif t \le \eta \intx r^2 u_t^2 \dif r \dif t +C(E_0+E^\frac{1}{2}(t) \int_0^t  \mathcal D(s)\dif s).
\end{align}
Now, multiplying the mass equation $\eqref{1.5}_1$ by $r^2 \rho_t$, one has
\begin{align*}
\intx r^2 \rho_t^2 \dif r\dif t =-\intx u\rho_r r^2 \rho_t \dif r \dif t -\intx \rho u_r r^2 \rho_t \dif r \dif t -\intx 2 r \rho u \rho_t \dif r \dif t\\
\le \frac{1}{2}\intx r^2 \rho_t^2 \dif r \dif t+C\intx r^2(u_r^2+2u^2)\dif r \dif t+E^\frac{1}{2}(t) \int_0^t  \mathcal D(s)\dif s,
\end{align*}
which,  combined with \eqref{1.33}, implies
\begin{align}\label{1.34}
\frac{1}{2} \intx r^2 \rho_t^2 \dif r \dif t \le C \eta \intx r^2 u_t^2\dif r\dif t +C(E_0+E^\frac{1}{2}(t) \int_0^t  \mathcal D(s)\dif s).
\end{align}

Multiplying the equation $\eqref{1.5}_2$ by $r^2 u_t$, we get{\small
\begin{align}\label{1.35}
\intx \rho r^2 u_t^2 \dif r\dif t=-\intx \rho u u_r r^2 u_t \dif r\dif t-\intx R \theta \rho_r r^2 u_t \dif r \dif t-\intx R \rho \theta_r r^2 u_t \dif r\dif t +\nonumber\\
\intx \left(\left(\frac{4}{3} \mu(\theta)+\lambda(\theta)\right)\left(u_{rr}+\frac{2}{r}u_r-\frac{2}{r^2} u\right)+\lambda^\prime(\theta)\theta_r\left(u_r+\frac{2}{r}u\right)+\frac{4}{3}\mu^\prime(\theta) \theta_r\left(u_r-\frac{u}{r}\right) \right)r^2 u_t.
\end{align}}

We estimate each term in \eqref{1.35} as follows.
First,  it is easy to see that
\begin{align*}
\intx \rho u u_r r^2 u_t \dif r \dif t \le  E^\frac{1}{2}(t) \int_0^t  \mathcal D(s)\dif s.
\end{align*}
Using the boundary condition $u|_{\partial \Omega}=0$, \eqref{1.33} and \eqref{1.34}, one has
\begin{align*}
&-\intx R\theta \rho_r r^2 u_t \dif r\dif t \\
&=-\int_0^t \frac{\dif}{\dif t} \int_{\Omega} R \theta \rho_r r^2 u \dif r \dif t +\intx R\theta_t \rho_r r^2 u \dif t \dif t +\intx R \theta \rho_{rt} r^2 u\dif r \dif t \\
&\le -\intx R\theta \rho_t (r^2 u)_r\dif r \dif t+C(E_0+E^\frac{1}{2}(t) \int_0^t  \mathcal D(s)\dif s)\\
&\le C \eta \intx r^2 u_t^2 \dif r \dif t +C(E_0+E^\frac{1}{2}(t) \int_0^t  \mathcal D(s)\dif s).
\end{align*}
Using \eqref{1.31}, we have
\begin{align*}
\intx R \rho \theta_r r^2 u_t \dif r \dif t \le \frac{1}{2} \int \rho  r^2 u_t^2 \dif r \dif t+C(E_0+E^\frac{1}{2}(t) \int_0^t  \mathcal D(s)\dif s).
\end{align*}
For the last term of \eqref{1.35}, one has
\begin{align*}
&\intx \left(\left(\frac{4}{3} \mu(\theta)+\lambda(\theta)\right)\left(u_{rr}+\frac{2}{r}u_r-\frac{2}{r^2} u\right)+\lambda^\prime(\theta)\theta_r\left(u_r+\frac{2}{r}u\right)\right.\\
&\qquad\qquad\qquad\qquad\left.+\frac{4}{3}\mu^\prime(\theta) \theta_r\left(u_r-\frac{u}{r}\right) \right)r^2 u_t \dif r \dif t\\
&\le \intx \left(\frac{4}{3} \mu(\theta)+\lambda(\theta)\right)\left((r^2u_r)_r-2u\right)u_t \dif r \dif t +CE^\frac{1}{2}(t) \int_0^t  \mathcal D(s)\dif s\\
&\le -\int_0^t \frac{\dif}{\dif t} \int_{\Omega} \left(\frac{4}{3} \mu(\theta)+\lambda(\theta)\right)(\frac{1}{2}r^2 u_r^2+u^2) \dif r \dif t +CE^\frac{1}{2}(t) \int_0^t  \mathcal D(s)\dif s\\
&\le C(E_0+E^\frac{1}{2}(t) \int_0^t  \mathcal D(s)\dif s).
\end{align*}
Now, we can choose $\eta$ sufficiently small and get
\begin{align}\label{1.36}
\intx r^2 u_t^2 \dif r \dif t \le C(E_0+E^\frac{1}{2}(t) \int_0^t  \mathcal D(s)\dif s).
\end{align}
This together with \eqref{1.33} and \eqref{1.34} imply
\begin{align}\label{1.37}
\intx ((r^2 u_r^2+2u^2)+r^2 \rho_t^2 )\dif r \dif t \le C(E_0+E^\frac{1}{2}(t) \int_0^t  \mathcal D(s)\dif s).
\end{align}

Until now, we have get the dissipation estimates of $r(u_t, u_r, \rho_t, \theta_r)$ and $u$.
Multiplying the momentum equation $\eqref{1.5}_2$ by $r^2 \rho_r$,  using \eqref{1.31}, \eqref{1.36} and mass equation $\eqref{1.5}_1$, we derive that{\small
\begin{align*}
&\intx R \theta r^2 \rho_r^2 \dif r \dif t=-\intx  \left( \rho u_t +\rho u u_r+R \rho \theta_r\right)r^2 \rho_r \dif r \dif t\\
&\qquad\qquad\qquad\qquad+\intx \left(\left(\frac{4}{3} \mu(\theta)+\lambda(\theta)\right)\left(u_{rr}+\frac{2}{r}u_r-\frac{2}{r^2} u\right)+\lambda^\prime(\theta)\theta_r\left(u_r+\frac{2}{r}u\right)\right.\\
&\qquad\qquad\qquad\qquad\qquad\qquad\qquad\qquad\qquad\qquad\qquad\qquad\left.+\frac{4}{3}\mu^\prime(\theta) \theta_r\left(u_r-\frac{u}{r}\right)\right) r^2 \rho_r \dif r \dif t\\
&\le \frac{1}{2}\intx  R\theta r^2 \rho_r^2 \dif r \dif t+ \intx \left(\frac{4}{3}\mu(\theta)+\lambda(\theta)\right)\left(u_{rr}+\frac{2}{r}u_r-\frac{2}{r^2} u\right) r^2 \rho_r \dif r \dif t\\
&\qquad\qquad\qquad\qquad +C(E_0+E^\frac{1}{2}(t) \int_0^t  \mathcal D(s)\dif s)\\
&= \frac{1}{2}\intx  R\theta r^2 \rho_r^2 \dif r \dif t+ \intx \left(\frac{4}{3}\mu(\theta)+\lambda(\theta)\right)\left(-\frac{1}{\rho}\right) \left( \rho_{tr}+u\rho_{rr}+2 \rho_r u_r+\frac{2}{r} u\rho_r\right) r^2 \rho_r \dif r \dif t \\ &\qquad\qquad\qquad\qquad+C(E_0+E^\frac{1}{2}(t) \int_0^t  \mathcal D(s)\dif s)\\
&\le \frac{1}{2}\intx  R\theta r^2 \rho_r^2 \dif r \dif t-\int_0^t \frac{\dif}{\dif t}\int_{\Omega}  \left(\frac{4}{3}\mu(\theta)+\lambda(\theta)\right)\frac{1}{2\rho} \rho_r^2 \dif r\dif t +C(E_0+E^\frac{1}{2}(t) \int_0^t  \mathcal D(s)\dif s)\\
&\le \frac{1}{2}\intx  R\theta r^2 \rho_r^2 \dif r \dif t+C(E_0+E^\frac{1}{2}(t) \int_0^t  \mathcal D(s)\dif s),
\end{align*}}
which gives
\begin{align}
\intx r^2 \rho_r^2 \dif r \dif t \le  C(E_0+E^\frac{1}{2}(t) \int_0^t  \mathcal D(s)\dif s). \label{1.38}
\end{align}
Finally, using the  energy equation $\eqref{1.5}_3$, \eqref{1.37} and Lemma \ref{le3}, we have
\begin{align}
\intx r^2 \theta_t^2 \dif r \dif t &\le \intx (r^2u_r^2+2u^2+r^2 q_r^2+2 q^2)\dif r \dif t+C E^\frac{1}{2}(t) \int_0^t  \mathcal D(s)\dif s \nonumber\\
&\le C(E_0+E^\frac{1}{2}(t) \int_0^t  \mathcal D(s)\dif s). \label{1.39}
\end{align}
 This finishes the proof of Lemma \ref{le5}.
\end{proof}

The following lemmas give the second-order a priori estimates.

 \begin{lemma}\label{le6}
 There exists a constant $C$ such that
 \begin{align}
 \int_{\Omega} \left( r^2 \rho_{tr}^2+(r^2 u_{tr}^2+2u_t^2)+r^2 \theta_{tr}^2 +\tau( r^2 q_{tr}^2+2q_t^2) \right) \dif r + \intx (r^2q_{tr}^2+2q_t^2)\dif r \dif t \nonumber\\
 +\left(\frac{4}{3} \mu+\lambda\right) \intx (r^2 u_{trr}^2+u_{tr}^2) \dif r \dif t\le C(E_0+E^\frac{1}{2}(t) \int_0^t  \mathcal D(s)\dif s+E^2(t)). \label{1.40}
 \end{align}
 \end{lemma}

 \begin{proof}
 Taking derivative with respect to $(t,r)$, one gets
 \begin{align}\label{1.41}
 \begin{cases}
 \rho_{ttr}+u \rho_{trr} +\rho u_{trr}+\frac{2}{r} \rho u_{rt}-\frac{2}{r^2} \rho u_t=H_1,\\
 \rho u_{ttr} +\rho u u_{trr}+p_{trr}-\left(\left(\frac{4}{3} \mu(\theta)+\lambda(\theta)\right)\left(u_{rr}+\frac{2}{r}u_r-\frac{2}{r^2} u\right)\right)_{tr}=H_2,\\
 \rho e_\theta \theta_{ttr}+ B(\rho, u, \theta, q) \theta_{trr}+p\left(u_r+\frac{2}{r} u\right)_{tr}+\left(q_r+\frac{2}{r} q\right)_{tr}=H_3,\\
 \tau(\theta) \rho (q_{ttr}+ u q_{trr})+q_{tr}+\kappa(\theta) \theta_{trr}=H_4,
 \end{cases}
 \end{align}
 where
 \begin{align*}
 &H_1:=-\left( \rho_t u_{rr}+2 \rho_r u_{rt}+2 u_r \rho_{tr}+u_t \rho_{rr}-\frac{2}{r^2}\rho_t u+\frac{2}{r} u\rho_{tr}+\frac{2}{r} \rho_t u_r+\frac{2}{r} \rho_r u_t\right),\\
& H_2:=-\left( \rho_t u_{tr}+\rho_r u_{tt}+u_t \rho_{tr}+(\rho u)_t u_{rr}+(\rho u)_r u_{tr}+(\rho u)_{tr} u_r\right)\\
&\qquad \qquad \qquad+\left(\lambda^\prime(\theta)\theta_r\left(u_r+\frac{2}{r}u\right)+\frac{4}{3}\mu^\prime(\theta) \theta_r\left(u_r-\frac{u}{r}\right)\right)_{tr},\\
 &H_3:=-\left((\rho e_\theta)_t \theta_{tr}+(\rho e_\theta)_r \theta_{tt}+(\rho e_\theta)_{tr}\theta_t+B_t \theta_{rr}+B_r \theta_{tr}+B_{tr}\theta_r+p_t(u_r+\frac{2}{r}u)_r\right.\\
 &\qquad \qquad \qquad\left.+p_r(u_r+\frac{2}{r} u)_t+p_{tr} \left((u_r+\frac{2}{r} u)\right)\right)+\left(\left(\frac{2}{\tau(\theta)}+\frac{4}{r} u\right) a(\theta) q^2\right)_{tr}\\
&\qquad \qquad \qquad+ \left(\mu(\theta)\left( 2u_r^2+\frac{4}{r^2}u^2-\frac{2}{3} \left(u_r+\frac{2}{r} u\right)^2\right)+\lambda(\theta)\left(u_r+\frac{2}{r} u\right)^2 \right)_{tr},\\
& H_4:=-\left( (\tau(\theta) \rho)_t q_{tr}+(\tau(\theta) \rho)_r q_{tt}+(\tau(\theta) \rho)_{tr} q_t+(\tau(\theta) \rho u)_t q_{rr}+(\tau(\theta \rho u)_r q_{tr}+(\tau(\theta) \rho u)_{tr} q_r \right.\\
& \qquad\qquad\left.+\kappa^\prime(\theta)(\theta_t \theta_{rr}+2\theta_r \theta_{tr})+\kappa^{\prime\prime} \theta_t \theta_r^2+\left( \tau(\theta) \rho \frac{2}{r} u q \right)_{tr}\right).
 \end{align*}

 Multiplying equation $\eqref{1.41}_2$ by $r^2 u_{tr}$, one gets
 \begin{align}\label{1.42}
& \int_{\Omega} \rho u_{ttr} r^2 u_{tr}\dif r  +\int_{\Omega} \rho u u_{trr} r^2 u_{tr}\dif r\nonumber\\
&+\int_{\Omega} \left(p_{trr} -\left(\left(\frac{4}{3} \mu(\theta)+\lambda(\theta)\right)\left(u_{rr}+\frac{2}{r}u_r-\frac{2}{r^2} u\right)\right)_{tr}\right) r^2 u_{tr}\dif r
 =\int_{\Omega} H_2 r^2 u_{tr}\dif r.
 \end{align}
 We estimate each term in the above equation as follows.

 First,
 \begin{align*}
 \int_{\Omega} \rho r^2 ( u_{ttr} u_{tr}+u  u_{trr} u_{tr}) \dif r =\int_{\Omega} \rho r^2 \left(  \left( \frac{1}{2} u_{tr}^2 \right)_t+u \left(\frac{1}{2} u_{tr}^2\right)_r\right) \dif r=\frac{\dif}{\dif t} \int_{\Omega} \frac{1}{2} \rho r^2 u_{tr}^2 \dif r.
 \end{align*}

 Second, using the boundary condition
 $$
 \left(p_{rt}-\left(\left(\frac{4}{3} \mu(\theta)+\lambda(\theta)\right)\left(u_{rr}+\frac{2}{r}u_r-\frac{2}{r^2} u\right)\right)_{t}\right)|_{\partial \Omega}=0,
 $$
 we have
 \begin{align*}
 \int_{\Omega} \left(p_{trr} -\left(\left(\frac{4}{3} \mu(\theta)+\lambda(\theta)\right)\left(u_{rr}+\frac{2}{r}u_r-\frac{2}{r^2} u\right)\right)_{tr}\right) r^2 u_{tr}\dif r  \\
 =-\int_{\Omega} \left(p_{tr} -\left(\left(\frac{4}{3} \mu(\theta)+\lambda(\theta)\right)\left(u_{rr}+\frac{2}{r}u_r-\frac{2}{r^2} u\right)\right)_{t}\right) (r^2 u_{tr})_r \dif r .
 \end{align*}
Using equation $\eqref{1.41}_1$, we have{\small
 \begin{align*}
 &-\int_{\Omega} p_{tr} (r^2 u_{tr})_r\dif r \\
 &=-\int_{\Omega} R\left( \theta \rho_{tr}+\rho \theta_{tr}+\theta_t \rho_r+\rho_t \theta_r\right) (r^2 u_{tr})_r \dif r\\
 &=-\int_{\Omega} R(\theta \rho_{tr}+\rho \theta_{tr}) (r^2 u_{tr})_r\dif r -\frac{\dif}{\dif t} \int_{\Omega} R(\theta_t \rho_r+\rho_t\theta_r)r^2u_{rr} \dif r+\int_{\Omega} R(\theta_t \rho_r+\rho_t\theta_r)_t r^2 u_{rr} \dif r\\
& \ge \int_{\Omega} R \theta \rho_{tr} \left(\frac{1}{\rho} r^2 \rho_{ttr}+\frac{u}{\rho} r^2 \rho_{trr}-2 u_t-\frac{1}{\rho} r^2 H_1\right) \dif r-\int_{\Omega} R\rho \theta_{tr} (r^2 u_{tr})_r  \\
&\qquad\qquad\qquad-\frac{\dif}{\dif t} \int_{\Omega} R(\theta_t \rho_r+\rho_t\theta_r)r^2u_{rr} \dif r- C E^\frac{1}{2}(t) \mathcal D(t)\\
& \ge \frac{\dif}{\dif t} \int_{\Omega} \frac{R\theta}{2\rho} r^2  \rho_{tr}^2 \dif r -\int_{\Omega} R\left( \left(\frac{\theta}{\rho} r^2\right)_t+\left( \frac{\theta}{\rho} u r^2 \right)_r \right) \frac{1}{2} \rho_{tr}^2 \dif r
 -\int_{\Omega} R\rho \theta_{tr} (r^2 u_{tr})_r \dif r \\
 &\qquad\qquad\qquad -\int_{\Omega} 2 R\theta \rho_{tr} u_t \dif r -\frac{\dif}{\dif t} \int_{\Omega} R(\theta_t \rho_r+\rho_t\theta_r)r^2u_{rr} \dif r - C E^\frac{1}{2}(t) \mathcal D(t)\\
& \ge  \frac{\dif}{\dif t} \int_{\Omega} \left(\frac{R\theta}{2\rho} r^2  \rho_{tr}^2  -R(\theta_t \rho_r+\rho_t\theta_r)r^2u_{rr}\right) \dif r   -\int_{\Omega} (R\rho \theta_{tr} (r^2 u_{tr})_r + 2 R\theta \rho_{tr} u_t )\dif r -  C E^\frac{1}{2}(t) \mathcal D(t).
 \end{align*}}
 On the other hand, using the boundary condition $u_t|_{\partial_\Omega}=0$, one has
 \begin{align*}
& \int_{\Omega}\left( \left(\frac{4}{3} \mu(\theta)+\lambda(\theta)\right)\left(u_{rr}+\frac{2}{r}u_r-\frac{2}{r^2} u\right)\right)_t (r^2 u_{tr})_r \dif r\\
  &\ge \int_{\Omega} \left(\frac{4}{3} \mu(\theta)+\lambda(\theta)\right)\left(u_{trr}+\frac{2}{r}u_{tr}-\frac{2}{r^2} u_t\right) (r^2 u_{trr}+2r u_{tr}) \dif r-CE^\frac{1}{2}(t) \mathcal D(t).\\
  &=\int_{\Omega} \left(\frac{4}{3} \mu(\theta)+\lambda(\theta)\right) \left( r^2 u_{trr}^2+4u_{tr}^2+4r u_{trr} u_{tr} -2 u_t u_{trr}-\frac{4}{r} u_t u_{tr}\right)\dif r-CE^\frac{1}{2}(t) \mathcal D(t).\\
  &\ge   \int_{\Omega} \left(\frac{4}{3} \mu(\theta)+\lambda(\theta)\right)\left(\frac{1}{5} r^2 u_{trr}^2+u_{tr}^2-\frac{2}{r^2} u_t^2\right) \dif r-C E^\frac{1}{2}(t) \mathcal D(t).
  \end{align*}
 Thus, we derive that
 \begin{align}\label{1.43}
 \frac{\dif }{\dif t} \int_{\Omega} \left(\frac{R\theta}{2\rho} r^2 \rho_{tr}^2-R(\theta_t \rho_r+\rho_t\theta_r)r^2u_{rr} +\frac{1}{2} \rho r^2 u_{tr}^2 \right) \dif r -\int_{\Omega} R \rho \theta_{tr} (r^2 u_{tr})_r\dif r \nonumber\\- \int_{\Omega} 2 R\theta \rho_{tr} u_t \dif r
 + \int_{\Omega} \left(\frac{4}{3} \mu(\theta)+\lambda(\theta)\right)\left(\frac{1}{5} r^2 u_{trr}^2+u_{tr}^2-\frac{2}{r^2} u_t^2\right) \dif r \le C E^\frac{1}{2}(t) \mathcal D(t).
 \end{align}
 Multiplying equation $\eqref{1.41}_3$ by $\frac{1}{\theta} r^2 \theta_{tr}$, and integrating the results, we obtain
 \begin{align}
 \int_{\Omega} \frac{\rho e_\theta}{\theta} \theta_{ttr} r^2 \theta_{tr} \dif r+\int_{\Omega} \frac{B}{\theta} \theta_{trr} r^2 \theta_{tr} \dif r+\int_{\Omega} R \rho\left( u_r+\frac{2}{r}u\right)_{rt} r^2 \theta_{tr} \dif r \nonumber \\
 +\int_{\Omega} \frac{1}{\theta} \left( q_r+\frac{2}{r} q\right)_{rt} r^2 \theta_{tr} \dif r =\int_{\Omega} H_3 \frac{1}{\theta} r^2 \theta_{tr} \dif r. \label{1.44}
 \end{align}
 We estimate each term in the above equation as follows.
 First,
 \begin{align*}
 \int_{\Omega} \frac{\rho e_\theta}{\theta} r^2 \theta_{ttr} \theta_{tr} \dif r  \ge \frac{\dif}{\dif t} \int_{\Omega} \frac{\rho e_\theta}{2\theta} r^2 \theta_{tr}^2 \dif r -C E^\frac{1}{2}(t) \mathcal D(t).
 \end{align*}
 Using the fact that $B|_{\partial \Omega}=0$, one has
 \begin{align*}
 \int_{\Omega} \frac{B}{\theta} \theta_{trr} r^2 \theta_{tr} \dif r =-\int_{\Omega} \left(\frac{B}{\theta} r^2\right)_r \frac{1}{2}\theta_{tr}^2 \dif r \ge -C E^\frac{1}{2}(t) \mathcal D(t).
 \end{align*}
 Second,
 \begin{align*}
 \int_{\Omega}  R \rho \left(u_r+\frac{2}{r}u\right)_{rt} r^2 \theta_{tr}\dif &r=\int_{\Omega} R \rho \left( u_{trr}+\frac{2}{r}u_{tr}-\frac{2}{r^2} u_t\right) r^2 \theta_{tr} \dif r\\
 &=\int_{\Omega} R \rho (r^2 u_{rt})_r \theta_{tr} \dif r-\int_{\Omega} 2 R \rho  u_t \theta_{tr}\dif r.
 \end{align*}
 Similarly, one get
 \begin{align*}
 \int_{\Omega} \frac{1}{\theta} \left(q_r+\frac{2}{r}q\right)_{rt} r^2 \theta_{tr}\dif r =\int_{\Omega} \frac{1}{\theta} (r^2 q_{tr})_r \theta_{tr}\dif r-\int_{\Omega} \frac{2}{\theta} q_t \theta_{tr}\dif r.
 \end{align*}
 One should pay attention to the term
 \[
  \int_{\Omega} H_3 \frac{1}{\theta} r^2 \theta_{tr} \dif r,
  \] since there will appear third-order terms. The typical  third-order term can be dealt with as follows:
  \begin{align*}
 & \int_{\Omega} \left(\left(\frac{4}{3}\mu(\theta)+\lambda(\theta)\right) u_r^2\right)_{tr} r^2 \theta_{tr} \dif r\\
&  \le C\int_{\Omega} \left(\frac{4}{3}\mu(\theta)+\lambda(\theta)\right) (u_r u_{trr}+u_{tr} u_{rr}) r^2 \theta_{tr}\dif r\\
 & \le \eta \int_{\Omega} \left(\frac{4}{3}\mu(\theta)+\lambda(\theta)\right) r^2u_{trr}^2 \dif r+ C(\eta) E^\frac{1}{2}(t) \mathcal D(t),
  \end{align*}
  where we use Young's inequality and the following estimate,
  \begin{align*}
  \| \left(\frac{4}{3}\mu(\theta)+\lambda(\theta)\right)  r^2u_{tr} u_{rr} \theta_{tr}\|_{L^1}&\le C \| \sqrt{\left(\frac{4}{3}\mu(\theta)+\lambda(\theta)\right) } u_{tr}\|_{L^\infty} \|ru_{rr}\|_{L^2} \|r\theta_{tr}\|_{L^2} \\
  &\le C E^\frac{1}{2} \mathcal D(t).
  \end{align*}
 Thus, we conclude that
 \begin{align}
 \frac{\dif}{\dif t} \int_{\Omega} \frac{\rho e_\theta}{2\theta} r^2 \theta_{tr}^2 \dif t+\int_{\Omega} R \rho (r^2 u_{tr})_r \theta_{tr} \dif r-\int_{\Omega} 2 R\rho u_t \theta_{tr} \dif r +\int_{\Omega} \frac{1}{\theta} (r^2 q_{rt})_r\theta_{tr}\dif r \nonumber\\
- \eta \int_{\Omega} \left(\frac{4}{3}\mu(\theta)+\lambda(\theta)\right) r^2u_{trr}^2 \dif r -\int_{\Omega} \frac{2}{\theta} q_t \theta_{tr} \dif r \le C E^\frac{1}{2}(t) \mathcal D(t). \label{1.45}
 \end{align}
 Multiplying equation $\eqref{1.41}_4$ by $\frac{1}{\kappa(\theta)\theta} r^2 q_{tr}$, one gets
 \begin{align}
 \int_{\Omega} \frac{\tau(\theta) }{\kappa(\theta)\theta} r^2 q_{tr} \rho (q_{ttr} + u q_{trr} )\dif r +\int_{\Omega} \frac{1}{\kappa(\theta)\theta}r^2 q_{tr}^2 \dif r+\int_{\Omega} \frac{1}{\theta} r^2 q_{tr} \theta_{trr} \dif r=\int_{\Omega} \frac{1}{\kappa(\theta)\theta} r^2 q_{tr} H_4.
 \end{align}
 We estimate again each term in the above equation separately.
First, using the mass equation, one has
 \begin{align*}
  \int_{\Omega} \frac{\tau(\theta)}{\kappa(\theta)\theta}  r^2  q_{tr} \rho( q_{ttr}+u q_{trr}) \dif r&= \int_{\Omega} \frac{\tau(\theta)}{\kappa(\theta)\theta} \rho r^2 ( (\frac{1}{2} q_{tr}^2)_t+u (\frac{1}{2}q_{tr}^2)_r) \dif r\\
  &\ge \frac{\dif}{\dif t} \int_{\Omega} \frac{\tau(\theta)}{\kappa(\theta)\theta} \rho r^2  \frac{1}{2} q_{tr}^2 \dif r -CE^\frac{1}{2}(t) \mathcal D(t).
  \end{align*}
 Second, using the fact that $\theta_{tr}|_{\partial \Omega}=0$, we have
 \begin{align*}
 \int_{\Omega} \frac{1}{\theta} r^2 q_{tr}\theta_{trr}   \dif r \ge -\int_{\Omega} \frac{1}{\theta} (r^2 q_{tr})_r  \theta_{tr} \dif r -CE^\frac{1}{2}(t) \mathcal D(t).
 \end{align*}
 Thus, we derive
 \begin{align}\label{1.47}
 \frac{\dif}{\dif t} \int_{\Omega} \frac{\tau(\theta)}{2\kappa(\theta)\theta} \rho r^2 q_{tr}^2 \dif r -\int_{\Omega} \frac{1}{\theta} (r^2 q_{tr})_r \theta_{tr} \dif r+\int_{\Omega} \frac{1}{\kappa(\theta)\theta} r^2 q_{tr}^2 \dif r \le CE^\frac{1}{2}(t) \mathcal D(t).
 \end{align}
Combining \eqref{1.43}, \eqref{1.45} and \eqref{1.47}, and choosing $\eta$ sufficiently small, we derive
\begin{align}
  &\frac{\dif}{\dif t} \int_{\Omega} \left( \frac{R\theta}{2\rho} r^2 \rho_{tr}^2 -R(\theta_t \rho_r+\rho_t\theta_r)r^2u_{rr}+\frac{1}{2}\rho r^2 u_{tr}^2+\frac{\rho e_\theta}{2\theta} r^2 \theta_{tr}^2+\frac{\tau(\theta)}{2\kappa(\theta)\theta} \rho r^2 q_{tr}^2 \right)\dif r \nonumber \\
  & -\int_{\Omega}  \left(2R u_t (\theta \rho_{tr}+\rho \theta_{tr})+\frac{2}{\theta} q_t\theta_{tr}\right)\dif r
+\int_{\Omega} \frac{1}{\kappa(\theta)\theta} r^2 q_{tr}^2 \dif r  \nonumber\\
&+ \int_{\Omega} \left(\frac{4}{3} \mu(\theta)+\lambda(\theta)\right)\left(\frac{1}{10} r^2 u_{trr}^2+u_{tr}^2-\frac{2}{r^2} u_t^2\right) \dif r \le CE^\frac{1}{2}(t) \mathcal D(t). \label{1.48}
 \end{align}
 From the momentum equation, one has
 \begin{align*}
 \rho u_{tt}+\rho u u_{tr}+p_{tr}=-(\rho_t u_t+(\rho u)_t u_r) \qquad\qquad\qquad\qquad\qquad\\
+ \left(\left(\frac{4}{3} \mu(\theta)+\lambda(\theta)\right)\left(u_{rr}+\frac{2}{r}u_r-\frac{2}{r^2} u\right)+\lambda^\prime(\theta)\theta_r\left(u_r+\frac{2}{r}u\right)+\frac{4}{3}\mu^\prime(\theta) \theta_r\left(u_r-\frac{u}{r}\right)\right)_t.
 \end{align*}
 Multiplying the above equation by $2 u_t$ yields{\small
 \begin{align}\label{1.49}
 \frac{\dif}{\dif t} \int_{\Omega} \rho u_t^2\dif r+\int_{\Omega} 2 u_t R(\rho \theta_{tr}+\theta \rho_{tr})\dif r
 +\int_{\Omega} \left(\frac{4}{3} \mu(\theta)+\lambda(\theta)\right)\left(2u_{tr}^2+\frac{2}{r^2} u_t^2 \right)\dif r \le CE^\frac{1}{2}(t) \mathcal D(t),
 \end{align}}
 where we use the fact that{\small
\begin{align*}
&\int_{\Omega} \left(\left(\frac{4}{3} \mu(\theta)+\lambda(\theta)\right)\left(u_{rr}+\frac{2}{r}u_r-\frac{2}{r^2} u\right)+\lambda^\prime(\theta)\theta_r\left(u_r+\frac{2}{r}u\right)+\frac{4}{3}\mu^\prime(\theta) \theta_r\left(u_r-\frac{u}{r}\right) \right)_t2 u_t \dif r\\
&\le -\int_{\Omega} \left(\frac{4}{3} \mu(\theta)+\lambda(\theta)\right)\left(2u_{tr}^2+\frac{2}{r^2} u_t^2 \right)\dif r+CE^\frac{1}{2}(t) \mathcal D(t).
\end{align*}}
 Similarly, taking the derivative with respect to $t$ in equation $\eqref{1.5}_4$, and multiplying the result by $\frac{2}{\kappa(\theta)\theta} q_t$, one obtains
 \begin{align}\label{1.50}
 \frac{\dif}{\dif t} \int_{\Omega} \frac{\tau(\theta)}{\kappa(\theta)\theta} \rho q_t^2 \dif r+\int_{\Omega} \frac{2}{\kappa(\theta)\theta} q_t^2\dif r+\int_{\Omega} \frac{2}{\theta} \theta_{rt}q_t \dif r \le CE^\frac{1}{2}(t) \mathcal D(t).
 \end{align}
 Therefore, by using \eqref{1.48}, \eqref{1.49} and \eqref{1.50}, one derives{\small
\begin{align}
  \frac{\dif}{\dif t} \int_{\Omega} \left(\frac{R\theta}{2\rho} r^2 \rho_{tr}^2 -R(\theta_t \rho_r+\rho_t\theta_r)r^2u_{rr} + \frac{1}{2}\rho( r^2 u_{tr}^2+2u_t^2)+\frac{\rho e_\theta}{2\theta} r^2 \theta_{tr}^2+\frac{\tau(\theta)}{2\kappa(\theta)\theta} \rho (r^2 q_{tr}^2+2q_t^2) \right)\dif r \nonumber\\
+ \int_{\Omega} \left(\frac{4}{3} \mu(\theta)+\lambda(\theta)\right)\left(\frac{1}{10} r^2 u_{trr}^2+3u_{tr}^2\right) \dif r +\int_{\Omega} \frac{1}{\kappa(\theta)\theta} (r^2 q_{tr}^2+2q_t^2) \dif r \le CE^\frac{1}{2}(t) \mathcal D(t). \label{1.51}
 \end{align}}
 Integrating this over $(0, t)$, we get the desired result. The proof of Lemm \ref{le6} is finished.
 \end{proof}
  Similarly, we can derive
 \begin{align}
  \frac{\dif}{\dif t} \int_{\Omega}  \left( \frac{1}{2}\rho r^2 u_{tt}^2+\frac{R\theta}{2\rho} r^2 \rho_{tt}^2+\frac{\rho e_\theta}{2\theta} r^2 \theta_{tt}^2+\frac{\tau(\theta)}{2\kappa(\theta)\theta} \rho r^2 q_{tt}^2 \right)\dif r \nonumber\\
+ \int_{\Omega} \left(\frac{4}{3} \mu(\theta)+\lambda(\theta)\right)\left(r^2 u_{ttr}^2\right) \dif r +\int_{\Omega} \frac{1}{\kappa(\theta)\theta} r^2 q_{tt}^2 \dif r \le \frac{C}{\tau} E^\frac{1}{2}(t) \mathcal D(t). \label{1.52}
 \end{align}
Note that the factor $\frac{1}{\tau}$ in \eqref{1.52} arises from the $\tau^2$ weight assigned to $\|r q_{tt}\|_{L^2}^2$ in the definition of $\mathcal{D}(t)$. Furthermore, motivated by the definition of $E(t)$, we integrate \eqref{1.52} and multiply the resulting equation by $\tau^2$ to establish the following lemma.
  \begin{lemma}\label{le7}
 There exists a constant $C$ such that
 \begin{align}
 \int_{\Omega} \tau^2 \left( r^2 \rho_{tt}^2+r^2 u_{tt}^2+r^2 \theta_{tt}^2 +\tau r^2 q_{tt}^2 \right) \dif r + \intx  \tau^2 r^2q_{tt}^2 \dif r \dif t\nonumber\\
 +   \left(\frac{4}{3} \mu +\lambda \right) \intx \left(r^2 u_{ttr}^2\right) \dif r \dif t
 \le C(E_0+E^\frac{1}{2}(t) \int_0^t  \mathcal D(s)\dif s). \label{1.53}
 \end{align}
 \end{lemma}
Furthermore, we have
 \begin{lemma}\label{le8}
 There exists a constant $C$ such that{\small
 \begin{align}
 \int_{\Omega} \left( r^2 \rho_{rr}^2+(r^2 u_{rr}^2+2u_r^2)+r^2 \theta_{rr}^2 +\tau( r^2 q_{rr}^2+2q_r^2) \right) \dif r
 \le C(E_0+E^\frac{1}{2}(t) \int_0^t  \mathcal D(s)\dif s+E^2(t)). \label{1.54}
 \end{align}}
 and
 \begin{align}\label{1.54-1}
 \int_{\Omega} \left(\frac{4}{3} \mu+\lambda\right)^2r^2 (u_{rrr}^2+\tau^2 u_{trr}^2) \dif r \le C(E_0+E^\frac{1}{2}(t) \int_0^t  \mathcal D(s)\dif s+E^2(t)).
 \end{align}
 \end{lemma}
\begin{proof}
From equation $\eqref{1.19}_1$, we have
\begin{align}\label{1.55}
\int_{\Omega} \rho^2 r^2 \left(u_{rr}+\frac{2}{r}u_r-\frac{2}{r^2} u\right)^2 \dif r =\int_{\Omega} r^2 \left(-\rho_{tr}-u \rho_{rr}+G_1\right)^2 \dif r.
\end{align}
For the left-hand side of the above equation, one has
\begin{align*}
&\int_{\Omega} \rho^2 r^2 \left( u_{rr}+\frac{2}{r}u_r-\frac{2}{r^2} u\right)^2 \dif r\\
&=\int_{\Omega} \rho^2 r^2\left( u_{rr}^2+\frac{4}{r^2} u_r^2+\frac{4}{r^4} u^2+\frac{4}{r} u_ru_{rr}-\frac{4}{r^2} u u_{rr}-\frac{8}{r^3} u u_r\right) \dif r\\
&=\int_{\Omega} \rho^2 \left( r^2 u_{rr}^2+8 u_r^2+4 u_r u_{rr}\right) \dif r \ge \frac{1}{4} \int_{\Omega} \left(\frac{1}{3} r^2 u_{rr}^2+2 u_r^2\right)\dif r ,
\end{align*}
where we used the $\varepsilon$-Young inequality
\[
4 r u_r u_{rr} \le \varepsilon (ru_{rr})^2+\frac{1}{4\varepsilon} (4u_r)^2,
\]
and take $\varepsilon:=\frac{2}{3}$.

For the right hand side of \eqref{1.54}, using Lemma \ref{le6}, one has
\begin{align*}
\int_{\Omega} r^2 \left(-\rho_{tr}-u \rho_{rr}+G_1\right)^2 \dif r  \le C(E_0+E^\frac{1}{2}(t) \int_0^t  \mathcal D(s)\dif s+E^2(t)),
\end{align*}
where we use the fact
\begin{align*}
\int_{\Omega} r^2 G_1^2 \dif r \le \int_{\Omega} r^2 (-2 \rho_r u_r-\frac{2}{r} \rho_r u)^2 \dif r \le CE^2(t).
\end{align*}
Thus, we derive that
\begin{align}\label{1.56}
\int_{\Omega} (r^2u_{rr}^2+u_r^2)\dif r \le C(E_0+E^\frac{1}{2}(t) \int_0^t  \mathcal D(s)\dif s+E^2(t)).
\end{align}
From equation $\eqref{1.19}_4$, using Lemmas \ref{le2}-\ref{le6}, one immediately gets
\begin{align}\label{1.57}
\int_{\Omega} \kappa^2(\theta) r^2 \theta_{rr}^2 \dif r &\le C \int_{\Omega} \left( \tau^2 r^2 q_{tr}^2+\tau^2 u^2 r^2 q_{rr}^2+r^2 q_r^2+r^2 G_4^2\right)\dif r \nonumber\\
 &\le C(E_0+E^\frac{1}{2}(t) \int_0^t  \mathcal D(s)\dif s+E^2(t)).
\end{align}
Multiplying the momentum equation $\eqref{1.19}_2$ by $r^2 \rho_{rr}$,  using \eqref{1.57}, Lemma \ref{le6} and the mass equation $\eqref{1.5}_1$, we derive{\small
\begin{align*}
&\int_{\Omega} R \theta r^2 \rho_{rr}^2 \dif r =-\int_{\Omega}  \left( \rho u_{tr} +\rho u u_{rr}+R \rho \theta_{rr}+2R\rho_r\theta_r-G_2\right)r^2 \rho_{rr} \dif r \\
&\qquad\qquad+\int_{\Omega} \left((\frac{4}{3} \mu(\theta)+\lambda(\theta))\left(u_{rr}+\frac{2}{r} u_r-\frac{2}{r^2} u\right)\right)_r r^2 \rho_{rr}\dif r\\
&\le \frac{1}{2}\int_{\Omega}  R\theta r^2 \rho_{rr}^2 \dif r + \int_{\Omega} \left(\frac{4}{3}\mu(\theta)+\lambda(\theta)\right)\left(u_{rr}+\frac{2}{r}u_r-\frac{2}{r^2} u\right)_r r^2 \rho_{rr} \dif r
 \\&\qquad\qquad+C(E_0+E^\frac{1}{2}(t) \int_0^t  \mathcal D(s)\dif s+E^2(t))\\
&= \frac{1}{2}\int_{\Omega}  R\theta r^2 \rho_{rr}^2 \dif r + \int_{\Omega} \left(\frac{4}{3}\mu(\theta)+\lambda(\theta)\right)\left(\left(-\frac{1}{\rho}\right) \left( \rho_{tr}+u\rho_{rr}
+2 \rho_r u_r+\frac{2}{r} u\rho_r\right)\right)_r r^2 \rho_{rr} \dif r \\
&\qquad\qquad+C(E_0+E^\frac{1}{2}(t) \int_0^t  \mathcal D(s)\dif s+E^2(t))\\
&\le \frac{1}{2}\int_{\Omega}  R\theta r^2 \rho_{rr}^2 \dif r- \frac{\dif}{\dif t}\int_{\Omega}  \left(\frac{4}{3}\mu(\theta)+\lambda(\theta)\right)\frac{1}{2\rho} \rho_{rr}^2 \dif r+C(E_0+E^\frac{1}{2}(t) \int_0^t  \mathcal D(s)\dif s+E^2(t)).
\end{align*}}
Therefore, Gronwall's inequality implies that
\begin{align}
\int_{\Omega} r^2 \rho_{rr}^2 \dif r \le C(E_0+E^\frac{1}{2}(t) \int_0^t  \mathcal D(s)\dif s+E^2(t)).
\end{align}
Finally, from equation $\eqref{1.19}_3$, one has
\begin{align} \label{1.58}
\int_{\Omega} r^2 \left(q_r+\frac{2}{r} q\right)_r^2 \dif r =\int_{\Omega} r^2 \left(G_3-\rho e_\theta \theta_{tr}-B \theta_{rr}-p\left(u_r+\frac{2}{r}u\right)_r \right)^2 \dif r.
\end{align}
Similar to \eqref{1.55}, we have
\begin{align*}
\int_{\Omega} r^2 \left(q_r+\frac{2}{r} q\right)_r^2 \dif r \ge \int_{\Omega} (\frac{1}{3} r^2 q_{rr}^2+2q_r^2) \dif r.
\end{align*}
 While, using \eqref{1.56} and Lemma \ref{le6}, one has
 \begin{align*}
& \int_{\Omega} r^2 \left(G_3-\rho e_\theta \theta_{tr}-B \theta_{rr}-p\left(u_r+\frac{2}{r}u\right)_r \right)^2 \dif r\\
& \le C\int_{\Omega} \left(r^2(G_3^2+\theta_{tr}^2+B^2\theta_{rr}^2)+r^2 u_{rr}^2+u_r^2\right)\dif r \\
& \le  C(E_0+E^\frac{1}{2}(t) \int_0^t  \mathcal D(s)\dif s+E^2(t)).
 \end{align*}
So, we derive that
\begin{align}\label{1.59}
\int_{\Omega} ( r^2 q_{rr}^2+q_r^2) \dif r \le C(E_0+E^\frac{1}{2}(t) \int_0^t  \mathcal D(s)\dif s+E^2(t)).
\end{align}
By virtue of \eqref{1.56}-\eqref{1.59}, the estimate \eqref{1.54} in Lemma \ref{le8} is verified. Furthermore, combining the second equations of \eqref{1.10} and \eqref{1.19} with Lemmas \ref{le2}-\ref{le6}, we directly derive inequality \eqref{1.54-1}. This completes the proof of Lemma \ref{le8}.
\end{proof}
 \begin{lemma}\label{le9}
There exists a constant $C$ such that
\begin{align}\label{1.60}
\intx \left( r^2 | D^2(\rho, u, \theta)|^2+r^2 q_{rr}^2+(u_r^2+u_t^2+q_r^2)\right)\dif r \dif t \le C(E_0+E^\frac{1}{2}(t) \int_0^t  \mathcal D(s)\dif s+E^2(t)).
\end{align}
\end{lemma}
\begin{proof}
From $\eqref{1.19}_4$ and using Lemma \ref{le6} and \ref{le3}, we have
\begin{align}\label{1.61}
\intx \kappa^2(\theta) r^2 \theta_{rr}^2 \dif r \dif t&=\intx r^2(G_4-\tau(\theta) \rho(q_{tr}+uq_{rr})-q_r)^2 \dif r \nonumber\\
& \le C(E_0+E^\frac{1}{2}(t) \int_0^t  \mathcal D(s)\dif s+E^2(t)).
\end{align}
Similarly, from equation $\eqref{1.10}_4$ and using Lemmas \ref{le2}, \ref{le6} and \ref{le7}, one gets
\begin{align}\label{1.62}
\intx r^2 \theta_{rt}^2 \dif r \dif t \le C(E_0+E^\frac{1}{2}(t) \int_0^t  \mathcal D(s)\dif s+E^2(t)).
\end{align}

Multiplying $\eqref{1.10}_3$ by $r^2(u_r+\frac{2}{r} u)_t$, using Lemma \ref{le6} and \eqref{1.62}, one has{\small
\begin{align*}
&\intx p r^2\left(u_r+\frac{2}{r} u\right)_t^2 \dif r\dif t \\
&=-\intx \left( \rho e_\theta \theta_{tt}+B\theta_{rt}+\left(q_r+\frac{2}{r}q\right)_t -F_3\right) r^2 \left(u_r+\frac{2}{r} u\right)_t \dif r \dif t\\
&\le - \intx  \rho e_\theta \theta_{tt}(r^2 u)_{rt} \dif r \dif t+ \frac{1}{2} \intx p r^2 \left(u_r+\frac{2}{r}u\right)_t^2 \dif r \dif t+C(E_0+E^\frac{1}{2}(t) \int_0^t  \mathcal D(s)\dif s+E^2(t)).
\end{align*}}

Using the boundary condition $u|_{\partial \Omega}=0$, \eqref{1.62} and Lemmas \ref{le2}, \ref{le6}, one has
\begin{align*}
&-\intx \rho e_\theta \theta_{tt} (r^2 u)_{rt} \dif r \dif t \\
&=-\int_0^t \left(\frac{\dif}{\dif t} \int_{\Omega} \rho e_\theta \theta_t (r^2 u)_{tr} \dif r \right) \dif t+\intx (\rho e_\theta)_t \theta_t (r^2 u)_{rt} \dif r \dif t +\intx \rho e_\theta \theta_t (r^2 u)_{ttr}\dif r \dif t\\
&\le -\intx \rho e_\theta \theta_{tr} r^2 u_{tt} \dif r \dif t +C(E_0+E^\frac{1}{2}(t) \int_0^t  \mathcal D(s)\dif s+E^2(t))\\
&\le \eta \intx r^2 u_{tt}^2 \dif r\dif t +C(\eta)(E_0+E^\frac{1}{2}(t) \int_0^t  \mathcal D(s)\dif s+E^2(t)).
\end{align*}

On the other hand, using the boundary condition $u_t|_{\partial_\Omega}=0$, we obtain
\begin{align*}
\intx p r^2 \left(u_r+\frac{2}{r}u\right)_t^2 \dif r \dif t=\intx p r^2(u_{rt}^2+\frac{4}{r} u_t u_{tr}+\frac{4}{r^2} u_t^2 ) \dif r \dif t \\
\ge \intx p(r^2 u_{tr}^2+2u_t^2) \dif x \dif t- CE^\frac{1}{2}(t)\mathcal D(t).
\end{align*}
Therefore, we derived that
\begin{align}\label{1.63}
\intx (r^2 u_{tr}^2+2u_t^2) \dif x \dif t \le \eta \intx r^2 u_{tt}^2 \dif r\dif t +C(\eta)(E_0+E^\frac{1}{2}(t) \int_0^t  \mathcal D(s)\dif s+E^2(t)).
\end{align}

Multiplying $\eqref{1.10}_2$ by $r^2 u_{tt}$,  using \eqref{1.62}, Lemma \ref{le6}, and the boundary condition $u_{tt}|_{\partial\Omega}=0$, one gets
\begin{align*}
&\intx \rho r^2 u_{tt}^2\dif r \dif t =-\intx \left( \rho u u_{tr}+R\rho \theta_{tr}+R\theta \rho_{tr}-F_2\right) r^2 u_{tt}\dif r \dif t\\
&\qquad\qquad+\intx (\frac{4}{3} \mu(\theta)+\lambda(\theta))\left(u_{rr}+\frac{2}{r} u_r-\frac{2}{r^2} u\right)_t r^2 u_{tt} \dif r \dif t\\
&\le -\intx R \theta \rho_{rt} r^2 u_{tt}\dif r \dif t +\frac{1}{2} \intx \rho r^2 u_{tt}^2\dif r \dif t +CE^\frac{1}{2}(t) \int_0^t  \mathcal D(s)\dif s\\
&\qquad\qquad -\int_0^t \frac{\dif}{\dif t} \int_{\Omega} \left(\frac{4}{3} \mu(\theta)+\lambda(\theta)\right)\left(\frac{1}{2} r^2u_{rt}^2+u_t^2\right) \dif r \dif t\\
&\le -\intx R \theta \rho_{rt} r^2 u_{tt}\dif r \dif t +\frac{1}{2} \intx \rho r^2 u_{tt}^2\dif r \dif t +C(E_0+E^\frac{1}{2}(t) \int_0^t  \mathcal D(s)\dif s+E^2(t)).
\end{align*}
Using equation $\eqref{1.10}_1$, Lemmas \ref{le6}-\ref{le7} and \eqref{1.63}, we have
\begin{align*}
&-\intx R\theta \rho_{rt} r^2 u_{tt}\dif r\dif t =\intx R\theta\left( (\rho u)_r+\frac{2}{r}\rho u\right)_r r^2 u_{tt}\dif r \dif t\\
&\le \intx R \theta \rho (u_{rr}+2u_r)r^2 u_{tt} \dif r \dif t+CE^\frac{1}{2}(t) \int_0^t  \mathcal D(s)\dif s\\
&\le\int_0^t \frac{\dif }{\dif t} \int_{\Omega} R\theta \rho (r^2u_r)_r u_t \dif r \dif t -\intx R\theta \rho(r^2u_r)_{rt} u_t \dif r \dif t+CE^\frac{1}{2}(t) \int_0^t  \mathcal D(s)\dif s\\
&\le \intx R\theta \rho (r^2 u_r)_t u_{rt} \dif r \dif t +C(E_0+E^\frac{1}{2}(t) \int_0^t  \mathcal D(s)\dif s+E^2(t))\\
&\le C \eta \intx r^2 u_{tt}^2 \dif r\dif t +C(\eta)(E_0+E^\frac{1}{2}(t) \int_0^t  \mathcal D(s)\dif s+E^2(t)).
\end{align*}
Therefore, by choosing $\eta$ sufficiently small, we get
\begin{align}\label{1.64}
\intx (r^2(u_{tr}^2+u_{tt}^2)+ u_t^2) \dif r \dif t \le C(E_0+E^\frac{1}{2}(t) \int_0^t  \mathcal D(s)\dif s+E^2(t)).
\end{align}
Multiplying the momentum equation $\eqref{1.10}_2$ by $r^2 \rho_{tr}$, using  \eqref{1.62}, \eqref{1.64} and Lemma \ref{le6}, one has
\begin{align*}
\intx R\theta r^2 \rho_{tr}^2 \dif r \dif t =&-\intx (\rho u_{tt}+\rho u u_{tr}+R \rho \theta_{tr}+F_2)r^2 \rho_{tr}\dif r \dif t\\
&+\intx (\frac{4}{3} \mu(\theta)+\lambda(\theta))\left(u_{rr}+\frac{2}{r} u_r-\frac{2}{r^2} u\right)_t r^2 \rho_{tr} \dif r \dif t\\
\le& \frac{1}{2} \intx R \theta r^2 \rho_{tr}^2\dif r \dif t +C(E_0+E^\frac{1}{2}(t) \int_0^t  \mathcal D(s)\dif s+E^2(t))\\
&-\intx \left(\frac{4}{3} \mu(\theta)+\lambda(\theta)\right) \frac{1}{\rho} (\rho_{tr}+u\rho_{rr}+2\rho_r u_r+\frac{2}{r} \rho_r u)_t r^2 \rho_{tr} \dif r \dif t \\
 \le& \frac{1}{2} \intx R \theta r^2 \rho_{tr}^2\dif r \dif t +C(E_0+E^\frac{1}{2}(t) \int_0^t  \mathcal D(s)\dif s+E^2(t))\\
&-\int_0^t \frac{\dif}{\dif t} \int_{\Omega} \left(\frac{4}{3} \mu(\theta)+\lambda(\theta)\right) \frac{1}{2\rho} r^2 \rho_{tr}^2 \dif r\dif t\\
\le& \frac{1}{2} \intx R \theta r^2 \rho_{tr}^2\dif r \dif t +C(E_0+E^\frac{1}{2}(t) \int_0^t  \mathcal D(s)\dif s+E^2(t)),
\end{align*}
which gives
\begin{align}\label{1.65}
\intx R \theta r^2 \rho_{tr}^2\dif r \dif t \le C(E_0+E^\frac{1}{2}(t) \int_0^t  \mathcal D(s)\dif s+E^2(t)).
\end{align}
In a similar way, we can derive that
\begin{align}\label{1.65b}
\intx R \theta r^2 \rho_{rr}^2\dif r \dif t \le C(E_0+E^\frac{1}{2}(t) \int_0^t  \mathcal D(s)\dif s+E^2(t)).
\end{align}
%
Similarly, $\eqref{1.10}_1$, $\eqref{1.19}_1$ together with \eqref{1.64} and \eqref{1.65} imply
\begin{align}\label{1.66}
\intx (r^2(u_{rr}^2+\rho_{tt}^2)+u_r^2) \dif r \dif t \le C(E_0+E^\frac{1}{2}(t) \int_0^t  \mathcal D(s)\dif s+E^2(t)).
\end{align}
Finally, using the energy equation $\eqref{1.19}_3$ and the above estimates, we have
\begin{align}
\intx (r^2 q_{rr}^2+q_r^2)\dif r\dif t &\le \intx (r^2(\theta_{tr}^2+u_{rr}^2)+u_r^2)\dif r\dif t+CE^\frac{1}{2}(t) \int_0^t  \mathcal D(s)\dif s \nonumber\\
&\le C(E_0+E^\frac{1}{2}(t) \int_0^t  \mathcal D(s)\dif s+E^2(t)).
\end{align}
Thus, the proof of Lemma \ref{le9} is finished.
\end{proof}
Now, combining Lemmas \ref{le1} - \ref{le9} with the choice
\[
\delta = \min \left\{ \delta_1,\  \left( \frac{1}{2C} \right)^{\!\!2},\  1 \right\}
\]
where \(C\) is the universal constant from these lemmas, we complete the proof of Proposition \ref{pro}.

\section{Proof of the main theorems}

\textbf{ Proof of Theorem \ref{th1.1}:}
By the definition of $E(0)$, we know that
\[
E(0)\le C_1\left( \|r(\rho_0-1, \theta_0-1, q_0)\|_{H^{3}}^2+\|ru_0\|_{H^4}^2\right)  \le C_1\epsilon_0
\]
with $C_1$ a universal constant independent of $\tau, \lambda, \mu$.
Now, choosing $\epsilon_0$ sufficiently small such that
\[
C_1 C \epsilon_0 <\frac{\delta}{2},
\]
then, by virtue of Proposition \ref{pro}, $E(t)<\frac{\delta}{2}$ which close the a priori assumption $E(t)\le \delta$. Therefore, the estimate \eqref{hu3.1} holds for the local solution $(\rho, u, \theta, q)$. By usual continuation methods, the local solution can be extended uniquely to the global one satisfying the estimate \eqref{hu3.1}. Thus, the proof of Theorem \ref{th1.1} is finished.

{\bf Proof of Theorem \ref{th1.2}  }: Fix $\tau>0$. Let $\epsilon=(\mu, \lambda)$ and $(\rho_\tau^\epsilon, u_\tau^\epsilon, \theta_\tau^\epsilon, q_\tau^\epsilon)$ be the global solutions obtained in Theorem \ref{th1.1}.  Then, we have
\begin{align} \label{4.2}
\sup_{0\le t <\infty} \|(\rho_\tau^\epsilon-1, u_\tau^\epsilon, \theta_\tau^\epsilon-1, q_\tau^\epsilon)\|_{H^2}^2+\int_0^\infty \| D(\rho_\tau^\epsilon, u_\tau^\epsilon, \theta_\tau^\epsilon, q_\tau^\epsilon)\|_{H^1}^2 \dif t \le C E(0),
\end{align}
where $C$ is a constant independent of $\epsilon$.
Thus, there exists $(\rho_\tau^0, u_\tau^0, \theta_\tau^0, q_\tau^0)\in L^\infty([0, \infty), H^2)$ such that
\begin{align*}
(\rho_\tau^\epsilon, u_\tau^\epsilon, \theta_\tau^\epsilon, q_\tau^\epsilon) \rightarrow  (\rho_\tau^0, u_\tau^0, \theta_\tau^0, q_\tau^0) \quad \text{weak}-\ast \quad \mbox{in } L^\infty([0, \infty), H^2)
\end{align*}
Furthermore, since $ \partial_t ( \rho_\tau^\epsilon, u_\tau^\epsilon, \theta_\tau^\epsilon, q_\tau^\epsilon) $ are bounded in $L^2([0,\infty); H^1)$, by classical compactness argument, $(\rho_\tau^\epsilon, u_\tau^\epsilon, \theta_\tau^\epsilon, q_\tau^\epsilon)$ are relatively compact in $C([0, T], H_{loc}^{2-\delta_0})$ for any $T>0$ and $0<\delta_0<2$. As a consequence, as $\epsilon \rightarrow 0$ and up to subsequences,
\begin{align*}
(\rho_\tau^\epsilon, u_\tau^\epsilon, \theta_\tau^\epsilon, q_\tau^\epsilon) \rightarrow  (\rho_\tau^0, u_\tau^0, \theta_\tau^0, q_\tau^0) \quad \text{strongly} \quad \mbox{in } C^0([0, T], H^{2-\delta_0})
\end{align*}
On the other hand, quantities involving $\epsilon=(\mu, \lambda)$, such as $\left( \left( \frac{4}{3} \mu(\theta) + \lambda(\theta) \right) \left( u_r + \frac{2}{r} u \right) \right)_r$, converge to zero almost everywhere as $\epsilon \to 0$. Therefore, it is sufficient to pass  to the limit in \eqref{1.5} and $(\rho_\tau^0, u_\tau^0, \theta_\tau^0, q_\tau^0)$ satisfies the following hyperbolized Euler-Cattaneo-Christov equations
\begin{align}\label{4.3}
\begin{cases}
	\rho_t+(\rho u)_r+\frac{2}{r} \rho u = 0, \\
	\rho u_t+\rho u u_r+ p_r = 0,\\
	\rho e_\theta \theta_t + (\rho u e_\theta - \frac{2a(\theta)}{Z(\theta)} q)\theta_r + p\left(u_r+\frac{2}{r} u\right) + q_r + \frac{2}{r} q =
	 \left(\frac{2}{\tau(\theta)}+\frac{4}{r} u\right) a(\theta) q^2,  \\
	\tau (\theta) \rho \left(q_t+u q_r+\frac{2}{r} u q\right) + q + \kappa(\theta) \theta_r = 0.
\end{cases}
\end{align}
This finishes the proof of Theorem \ref{th1.2}.

{\bf Proof of Theorem\ref{th1.3}}: Fix $\epsilon=(\mu, \lambda)$. Let  $(\rho_\tau^\epsilon, u_\tau^\epsilon, \theta_\tau^\epsilon, q_\tau^\epsilon)$ be the global solutions obtained in Theorem \ref{th1.1}.  Then, we have
\begin{align} \label{4.4}
\sup_{0\le t <\infty} \|(\rho_\tau^\epsilon-1, \theta_\tau^\epsilon-1, q_\tau^\epsilon)\|_{H^2}^2+\int_0^\infty \| D(\rho_\tau^\epsilon,  \theta_\tau^\epsilon, q_\tau^\epsilon)\|_{H^1}^2 \dif t \le C E(0),
\end{align}
and
\begin{align}\label{4.4-1}
\| u_\tau^\epsilon\|_{H^3}^2+\int_0^\infty  \|D u_\tau^\epsilon\|_{H^2}^2 \dif t \le C E(0),
\end{align}
where $C$ is a constant independent of $\tau$ but possible depends on $\epsilon$.
Thus, there exists $(\rho_0^\epsilon-1, \theta_0^\epsilon-1, q_0^\epsilon)\in L^\infty([0, \infty), H^2)$ and $u_0^\epsilon \in L^\infty([0, \infty), H^3)$ such that
\begin{align*}
(\rho_\tau^\epsilon,  \theta_\tau^\epsilon, q_\tau^\epsilon) \rightarrow  (\rho_0^\epsilon,  \theta_0^\epsilon, q_0^\epsilon) \quad \text{weak}-\ast \quad \mbox{in } L^\infty([0, \infty), H^2)
\end{align*}
and
\begin{align*}
 u_\tau^\epsilon  \rightarrow    u_0^\epsilon  \quad \text{weak}-\ast \quad \mbox{in } L^\infty([0, \infty), H^3).
\end{align*}
Furthermore, since $ \partial_t ( \rho_\tau^\epsilon, \theta_\tau^\epsilon, q_\tau^\epsilon) $  and $\partial_t u_\tau^\epsilon$  are bounded in $L^2([0,\infty); H^1)$ and $L^2([0,\infty); H^2)$, respectively, by classical compactness argument, $(\rho_\tau^\epsilon,  \theta_\tau^\epsilon, q_\tau^\epsilon)$ and $u_\tau^\epsilon$ are relatively compact in $C([0, T], H_{loc}^{2-\delta_0})$ and $C([0, T], H_{loc}^{3-\delta_0})$, respectively,   for any $T>0$ and $0<\delta_0<2$. As a consequence, as $\tau \rightarrow 0$ and up to subsequences,
\begin{align*}
(\rho_\tau^\epsilon, \theta_\tau^\epsilon, q_\tau^\epsilon) \rightarrow  (\rho_0^\epsilon,  \theta_0^\epsilon, q_0^\epsilon) \quad \text{strongly} \quad \mbox{in } C^0([0, T], H_{loc}^{2-\delta_0})
\end{align*}
and
\begin{align*}
  u_\tau^\epsilon \rightarrow  u_0^\epsilon  \quad \text{strongly} \quad \mbox{in } C^0([0, T], H_{loc}^{3-\delta_0})
\end{align*}
 On the other hand, as $\tau \rightarrow 0$, we have
\[
\tau(\theta_\tau^\epsilon) \rho_\tau^\epsilon  (\partial_t q_\tau^\epsilon+u_\tau^\epsilon \partial_r q_\tau^\epsilon+\frac{2}{r} u_\tau^\epsilon q_\tau^\epsilon) \rightharpoonup 0\quad \mathrm{in} \,\; \mathcal {D}^\prime((0,\infty)\times \Omega),
\]
which gives $q_0^\epsilon=-\kappa(\theta_0^\epsilon) \partial_r\theta_0^\epsilon, a.e.$.
Therefore, it is sufficient to pass to the limit in \eqref{1.5} and $(\rho_0^\epsilon, u_0^\epsilon, \theta_0^\epsilon) $ satisfies the  compressible Navier-Stokes-Fourier equations in spherical symmetry as
\begin{align}\label{4.5}
\begin{cases}
	\rho_t+(\rho u)_r+\frac{2}{r} \rho u = 0, \\
	\rho u_t+\rho u u_r+ p_r = \left( \left(\frac{4}{3} \mu(\theta)+\lambda(\theta)\right)\left(u_{r}+\frac{2}{r}u \right)\right)_r ,\\
	\rho C_v \theta_t + \rho u C_v \theta_r + p\left(u_r+\frac{2}{r} u\right) -\frac{1}{r^2} (r^2 \kappa(\theta) \theta_r)_r \\
	\qquad\qquad\qquad= \mu(\theta)\left( 2u_r^2+\frac{4}{r^2}u^2 - \frac{2}{3} \left(u_r+\frac{2}{r} u\right)^2\right) + \lambda(\theta)\left(u_r+\frac{2}{r} u\right)^2.
 \end{cases}
\end{align}
This finishes the proof of Theorem \ref{th1.3}.

\end{document}